\documentclass[hidelinks,onefignum,onetabnum]{siamart250211}

\ifpdf
\hypersetup{
	pdftitle={The Generalized Matrix Norm Problem},
	pdfauthor={A. Kulmburg}
}
\fi

\headers{The Generalized Matrix Norm Problem}{A. Kulmburg}

\title{The Generalized Matrix Norm Problem\thanks{This work has been accepted for publication in the \emph{SIAM Journal on Matrix Analysis and Applications}. The final version is available at \url{https://doi.org/10.1137/23M1605545}
		\funding{This work was funded by the European Research Council (ERC) under grant no.~817629 (project justITSELF).}}}

\author{Adrian Kulmburg\thanks{School of Computation, Information and Technology, Technical University of Munich
		(\email{adrian.kulmburg@tum.de}, \url{ce.cit.tum.de/cps/members/adrian-kulmburg-msc/}).}}

\usepackage{xparse}
\usepackage{bm}
\usepackage{amsopn}

\usepackage{amsfonts}
\usepackage{amssymb}

\usepackage{pgfplots}
\pgfplotsset{compat=newest}
\usetikzlibrary{plotmarks}
\usetikzlibrary{arrows.meta}
\usepgfplotslibrary{patchplots}
\usepackage{grffile}
\usepackage{amsmath}

\usepackage{tikz}
\usetikzlibrary{decorations.markings,arrows.meta,decorations.pathreplacing,arrows,shapes}
\usetikzlibrary{patterns}
\usetikzlibrary{automata,positioning}
\usepackage{pgfplots}
\usepgfplotslibrary{groupplots}
\usepackage{multirow}

\usepackage{xpatch}
\xapptocmd\bfseries{\boldmath}{}{}

\newcommand{\naturals}{\mathbb{N}}
\newcommand{\reals}{\mathbb{R}}

\newcommand{\Hset}{\mathcal{H}}
\newcommand{\Vset}{\mathcal{V}}
\newcommand{\Wset}{\mathcal{W}}
\newcommand{\Xset}{\mathcal{X}}
\newcommand{\Yset}{\mathcal{Y}}

\newcommand{\ball}{\mathcal{B}}

\DeclareMathOperator{\trace}{Tr}

\DeclareMathOperator{\sign}{sign}
\DeclareMathOperator{\diag}{diag}
\DeclareMathOperator{\Diag}{Diag}
\DeclareMathOperator{\Image}{Im}

\renewcommand{\v}[1]{\vec{#1}}
\newcommand{\mat}[1]{\bm{\underline{\smash{#1}}}}

\newcommand{\mean}{\mathbb{E}}
\newcommand{\grand}{\mathsf{g}}
\newcommand{\xrand}{\mathsf{x}}

\newcommand{\problem}[1]{\mathsf{#1}}
\newcommand{\Poly}{\problem{P}}
\newcommand{\NP}{\problem{NP}}
\newcommand{\APX}{\problem{APX}}

\newcommand{\Set}[2]{\left\{#1\;\middle|\;#2\right\}}

\NewDocumentCommand{\iprod}{gg}{\left\langle\IfValueTF{#1}{#1}{\cdot},\IfValueTF{#2}{#2}{\cdot}\right\rangle}
\NewDocumentCommand{\absiprod}{gg}{\left|\left\langle\IfValueTF{#1}{#1}{\cdot},\IfValueTF{#2}{#2}{\cdot}\right\rangle\right|}
\NewDocumentCommand{\norm}{g}{\left\|\IfValueTF{#1}{#1}{\cdot}\right\|}
\NewDocumentCommand{\nnorm}{g}{\|\IfValueTF{#1}{#1}{\cdot}\|}

\newsiamthm{claim}{Claim}
\newsiamremark{remark}{Remark}
\newsiamremark{example}{Example}

\makeatletter
\DeclareRobustCommand\widecheckinternal[1]{{\mathpalette\@widecheckinternal{#1}}}
\def\@widecheckinternal#1#2{%
	\setbox\z@\hbox{\m@th$#1#2$}%
	\setbox\tw@\hbox{\m@th$#1%
		\widehat{\vrule\@width\z@\@height\ht\z@
			\vrule\@height\z@\@width\wd\z@}$}%
	\dp\tw@-2\ht\z@
	\@tempdima\ht\z@ \advance\@tempdima2\ht\tw@ \divide\@tempdima\thr@@
	\setbox\tw@\hbox{\raise1.05\@tempdima\hbox{\scalebox{1}[-1]{\lower\@tempdima\box\tw@}}}%
	{\ooalign{\box\tw@ \cr \box\z@}}}
\makeatother
\newcommand{\widecheck}[1]{\,\widecheckinternal{\kern -2pt #1}}

\definecolor{Blue}{RGB}{0,92,171}
\definecolor{Red}{RGB}{227,27,35}
\definecolor{Yellow}{RGB}{255,195,37}
\definecolor{YellowDark}{RGB}{219, 126, 8}

\newcommand{\padding}{0.3}

\pgfdeclarepatternformonly{my large dots}{\pgfqpoint{-1pt}{-1pt}}{\pgfqpoint{5pt}{5pt}}{\pgfqpoint{5pt}{5pt}}%
{
	\pgfpathcircle{\pgfqpoint{3pt}{3pt}}{2pt}
	\pgfusepath{fill}
}

\newcommand{\FPTAS}{\mathsf{FPTAS}}

\newcommand{\figOneShift}{1.5}

\begin{document}
	
	\maketitle
	
	\begin{abstract}
		We study the computability of the operator norm of a matrix with respect to norms induced by linear operators. Our findings reveal that this problem can be solved in polynomial time in certain situations, and we discuss how it can be approximated in other cases. Along the way, we investigate the concept of push-forward and pull-back of seminorms, which leads us to uncover novel duality principles that come into play when optimizing over the unit ball of norms.
	\end{abstract}
	
	
	
	\section{Introduction}
	One way to assess the effect of a linear operator $L : \Vset \rightarrow \Wset$ between normed vector spaces $(\Vset, \norm_\Vset)$ and $(\Wset, \norm_\Wset)$ is to compute the operator norm
	\begin{equation}
		\label{eq:basic_problem}
		\norm{L} := \sup_{\norm{v}_{\Vset}\leq 1}\norm{L(v)}_{\Wset}.
	\end{equation}
	This has a wide range of applications, depending on the norms involved. One particular class of norms for which this problem has been studied extensively \cite{Bhaskara2011,Bhattiprolu2018,Grothendieck1956,Nesterov1998,Steinberg2005} are the vector $p$-norms on $\reals^n$ for $p\in[1,\infty]$. In that case, \eqref{eq:basic_problem} can be formulated as in \cite[Chapter 5.6]{Horn2012} for a matrix $\mat{A} \in \reals^{n\times m}$:
	\begin{equation}
		\label{eq:matrix_problem}
		\norm{\mat{A}}_{p\mapsto q} := \max_{\norm{\v{v}}_p\leq 1} \norm{\mat{A}\v{v}}_q.
	\end{equation}
	Computing $\norm{\mat{A}}_{p\mapsto q}$ is known to be $\NP$-hard for most choices of $p$ and $q$ \cite{Bhaskara2011,Bhattiprolu2023}, although constant-ratio approximations occasionally do exist \cite{Bhattiprolu2018,Steinberg2005}. The problem is only known to be solvable in polynomial time for $p=q=2$ or when $p=1$ or $q = \infty$ (see also \cite{Steinberg2005}).
	The vector $p$-norms possess a lot of symmetries, and one might be interested in investigating more advanced norms. A simple generalization is to consider norms of the form
	\begin{equation*}
		\v{v} \mapsto \norm{\mat{M}\v{v}}_p,
	\end{equation*}
	where $\mat{M}\in\reals^{m\times n}$ is a matrix and $p\in[1,\infty]$. As we will see in Section \ref{sec:push-forward_pull-back}, this is a norm if $\mat{M}$ represents an injective operator, and so a generalization of \eqref{eq:matrix_problem} is
	\begin{equation}
		\label{eq:matrix_problem_generalized}
		\norm{\mat{A}}_{p\mapsto q;\mat{B},\mat{C}} := \max_{\norm{\mat{B}\v{v}}_p\leq 1} \norm{\mat{C}\,\mat{A}\v{v}}_q,
	\end{equation}
	for matrices $\mat{B}$ and $\mat{C}$. Of course, computing \eqref{eq:matrix_problem_generalized} is as hard as computing
	\begin{equation}
		\label{eq:matrix_problem_calm_down}
		\norm{\mat{A}}_{p\mapsto q;\mat{B}} := \max_{\norm{\mat{B}\v{v}}_p\leq 1} \norm{\mat{A}\v{v}}_q,
	\end{equation}
	since $\norm{\mat{A}}_{p\mapsto q;\mat{B},\mat{C}} = \norm{\mat{C}\,\mat{A}}_{p\mapsto q;\mat{B}}$. Therefore, for our analysis, we can focus on \eqref{eq:matrix_problem_calm_down}, which we call the \emph{generalized matrix $p\mapsto q; \mat{B}$-norm} of $\mat{A}$ (as opposed to the \emph{classical} $p\mapsto q$-norm from \eqref{eq:matrix_problem}).
	
	One key aspect used in \cite{Bhattiprolu2018,Steinberg2005} to analyze $\norm{\mat{A}}_{p\mapsto q}$ is the fact that
	\begin{equation*}
		\norm{\mat{A}}_{p\mapsto q} = \norm{\mat{A}^\top}_{q^*\mapsto p^*},
	\end{equation*}
	where $p^*$ and $q^*$ are the Hölder conjugates of $p$ and $q$, respectively. This allows us to concentrate on the cases where $1/p + 1/q \geq 1$. Unfortunately, such a symmetry does not exist for $\norm{\mat{A}}_{p\mapsto q; \mat{B}}$, as we will see in Example \ref{ex:dual_B_p_not_p_star}. This lead us to new ways to reformulate $\norm{\mat{A}}_{p\mapsto q; \mat{B}}$, which revealed a deep connection to the containment problem (that is, checking whether a set is contained in another) for ellipsotopes \cite{Kousik2022} (see also Example \ref{rmk:ellipsotopes}), which is treated in the companion paper \cite{Kulmburg2024} where we analyze in detail the computational complexity of the containment problem. An overview of the resulting complexity of computing $\norm{\mat{A}}_{p\mapsto q; \mat{B}}$ for different values of $p$ and $q$ can be found in Fig. \ref{fig:hardness_square}. Note that, for clarity, in \cite{Kulmburg2024} the notation is chosen so that one considers the containment problem of a $p$-ellipsotope in a $q$-ellipsotope. Therefore, for the remainder of the present article, whenever the reader is directed to a result from \cite{Kulmburg2024}, the values of $p$ and $q$ in the present document should be transformed into $p\gets q^*$ and $q \gets p^*$ for \cite{Kulmburg2024}.
	
	\begin{figure}[t!]
		\centering
			\begin{tikzpicture}[scale=0.5]
			
			\draw[-] (0,0) -- (10,0) -- (10,10) -- (0,10) -- (0,0);
			
			\draw[-] (0,10) -- (-0.1-\padding,10);
			\draw[-] (0,5) -- (-0.1-\padding,5);
			\draw[-] (0,0) -- (-0.1-\padding,0);
			
			\draw[-] (0,0) -- (0,-0.1-\padding);
			\draw[-] (5,0) -- (5,-0.1-\padding);
			\draw[-] (10,0) -- (10,-0.1-\padding);
			
			\draw[->] (0,10) -- (0,11);
			\draw[->] (10,0) -- (11,0);
			
			\fill[pattern=my large dots, pattern color=Red] (0,0) rectangle (5,5);
			\fill[color=Red, opacity=0.3] (0,0) rectangle (5,5);
			\fill[color=Red] (-\padding/2,-\padding/2) rectangle (5-\padding/2,\padding/2);
			
			\fill[pattern=my large dots, pattern color=Red] (5,5) rectangle (10,10);
			\fill[color=Red, opacity=0.3] (5,5) rectangle (10,10);
			\fill[color=Red] (10-\padding/2,5+\padding/2) rectangle (10+\padding/2,10+\padding/2);
			
			\fill[pattern=crosshatch, pattern color=Blue] (5,0) rectangle (10,5);
			\fill[color=Blue, opacity=0.3] (5,0) rectangle (10,5);
			\fill[color=Blue] (5-\padding/2,-\padding/2) rectangle (5+\padding/2,5+\padding/2);
			\fill[color=Blue] (5-\padding/2,5-\padding/2) rectangle (10+\padding/2,5+\padding/2);
			\fill[color=Blue] (5-\padding/2,-\padding/2) rectangle (10+\padding/2,\padding/2);
			\fill[color=Blue] (10-\padding/2,-\padding/2) rectangle (10+\padding/2,5+\padding/2);
			
			\fill[color=Yellow] (-\padding/2,10-\padding/2) rectangle (10+\padding/2,10+\padding/2);
			\fill[color=Yellow] (5,5) circle (\padding);
			
			\fill[color=black] (-\padding/2,5-\padding/2) rectangle (\padding/2,10-\padding/2);
			\fill[color=Red] (-\padding/2,-\padding/2) rectangle (\padding/2,5-\padding/2);
			
			\draw (0,11.5) node {$\bm{q}$};
			\draw (11.5,0) node {$\bm{p}$};
			
			\draw (-0.5-\padding,10) node {$\infty$};
			\draw (-0.5-\padding,5) node {$2$};
			\draw (-0.5-\padding,0) node {$1$};
			
			\draw (10,-0.5-\padding) node {$\infty$};
			\draw (5,-0.5-\padding) node {$2$};
			\draw (0,-0.5-\padding) node {$1$};

			\fill[color=Yellow] (-3,-3.5) rectangle (-2,-4);
			\fill[color=black] (-3,-4.5) rectangle (-2,-5);

			\draw (-2,-3.7) node[anchor=west] { : In $\FPTAS$};
			\draw (-2,-4.7) node[anchor=west] { : Not in $\FPTAS$};

			\fill[pattern=crosshatch, pattern color=Blue] (5,-3.5) rectangle (6,-4);
			\fill[color=Blue, opacity=0.3] (5,-3.5) rectangle (6,-4);
			\fill[pattern=my large dots, pattern color=Red] (5,-4.5) rectangle (6,-5);
			\fill[color=Red, opacity=0.3] (5,-4.5) rectangle (6,-5);
			\draw (6,-3.7) node[anchor=west] { : $\tau$-inapproximable};
			\draw (6,-4.7) node[anchor=west] { : Inapproximable};

		\end{tikzpicture}
		
		\caption{An overview of the hardness results presented in \cite{Kulmburg2024}, adapted to reflect the complexity of computing generalized $p\mapsto q;\mat{B}$-norms. The complexity class $\FPTAS$ is the class of problems that can be approximated within arbitrary accuracy in polynomial time with respect to the representation size of $\mat{A}$ and $\mat{B}$ and the accuracy parameter. A solid boundary line means that the boundary is included in the region. For more information, we refer to \cite{Kulmburg2024}. This graph was inspired by \cite[Fig.~1.1.]{Bhattiprolu2023}.}
		\label{fig:hardness_square}
	\end{figure}
	
	The paper is organized as follows: After introducing some notation, Section \ref{sec:prelims} reviews some core concepts such as duality and extended seminorms. We continue in Section \ref{sec:push-forward_pull-back} by addressing the concepts of push-forward and pull-back of a norm and reveal that these two notions are dual to each other. We then turn our attention to the generalized matrix norm problem by discussing instances where $\norm{\mat{A}}_{p\mapsto q;\mat{B}}$ can be computed in polynomial time in Section \ref{sec:tractable}, while in Section \ref{sec:approximable} we treat instances where the problem can be approximated (a graphical overview of our results is shown in Fig. \ref{fig:approximability_square}). Finally, in Section \ref{sec:numerical_evaluations}, we numerically verify the precision of some of the approximations from Section \ref{sec:approximable}.
	
	\begin{figure}[t!]
		\centering
			\begin{tikzpicture}[scale=0.5]
			
			\draw[-] (0,0) -- (10,0) -- (10,10) -- (0,10) -- (0,0);
			
			\draw[-] (0,10) -- (-0.1-\padding,10);
			\draw[-] (0,5) -- (-0.1-\padding,5);
			\draw[-] (0,0) -- (-0.1-\padding,0);
			
			\draw[-] (0,0) -- (0,-0.1-\padding);
			\draw[-] (5,0) -- (5,-0.1-\padding);
			\draw[-] (10,0) -- (10,-0.1-\padding);
			
			\draw[->] (0,10) -- (0,11);
			\draw[->] (10,0) -- (11,0);
			
			\fill[pattern=crosshatch, pattern color=Blue] (5,0) rectangle (10-\padding/2,5);
			\fill[color=Blue, opacity=0.3] (5,0) rectangle (10-\padding/2,5);
			\fill[color=Blue] (5-\padding/2,-\padding/2) rectangle (5+\padding/2,5+\padding/2);
			\fill[color=Blue] (5-\padding/2,5-\padding/2) rectangle (10-\padding/2,5+\padding/2);
			\fill[color=black] (5-\padding/2,-\padding/2) rectangle (10-\padding/2,\padding/2);
			
			\fill[color=Yellow] (-\padding/2,10-\padding/2) rectangle (10+\padding/2,10+\padding/2);
			\fill[color=Yellow] (5,5) circle (\padding);
			
			\fill[color=Red] (-\padding/2,-\padding/2) rectangle (5-\padding/2,\padding/2);
			
			\draw (0,11.5) node {$\bm{q}$};
			\draw (11.5,0) node {$\bm{p}$};
			
			\draw (-0.5-\padding,10) node {$\infty$};
			\draw (-0.5-\padding,5) node {$2$};
			\draw (-0.5-\padding,0) node {$1$};
			
			\draw (10,-0.5-\padding) node {$\infty$};
			\draw (5,-0.5-\padding) node {$2$};
			\draw (0,-0.5-\padding) node {$1$};

			\fill[color=Yellow] (-3-\figOneShift,-3.5) rectangle (-2-\figOneShift,-4);
			\fill[color=black] (-3-\figOneShift,-4.5) rectangle (-2-\figOneShift,-5);

			\draw (-2-\figOneShift,-3.7) node[anchor=west] { : Tractable};
			\draw (-2-\figOneShift,-4.7) node[anchor=west] { : $\gamma_p/\gamma_1$-approximable};

			\fill[pattern=crosshatch, pattern color=Blue] (7-\figOneShift,-3.5) rectangle (8-\figOneShift,-4);
			\fill[color=Blue, opacity=0.3] (7-\figOneShift,-3.5) rectangle (8-\figOneShift,-4);
			\fill[color=Red] (7-\figOneShift,-4.5) rectangle (8-\figOneShift,-5);
			\fill[color=Red, opacity=0.3] (7-\figOneShift,-4.5) rectangle (8-\figOneShift,-5);
			\draw (8-\figOneShift,-3.7) node[anchor=west] { : $\gamma_p\gamma_{q^*}$-approximable};
			\draw (8-\figOneShift,-4.7) node[anchor=west] { : $C(p,l,m)$-approximable};

		\end{tikzpicture}
		
		\caption{An overview of our results from Sections \ref{sec:tractable} and \ref{sec:approximable} for the computability/approximability of generalized matrix norms. A solid boundary line means that the boundary is included in the region. Here, $C(p,l,m) = \min\{\gamma_p\sqrt{m},l^{1/p-1/2}\}/\gamma_1$, where $m$ and $l$ are the number of rows of $\mat{A}$ and $\mat{B}$, respectively, and $\gamma_s$ for $s\in [1,\infty)$ is the $s$-th root of the $s$-th absolute moment of a Gaussian random variable. This graph was inspired by \cite[Fig.~1.1.]{Bhattiprolu2023}.}
		\label{fig:approximability_square}
	\end{figure}
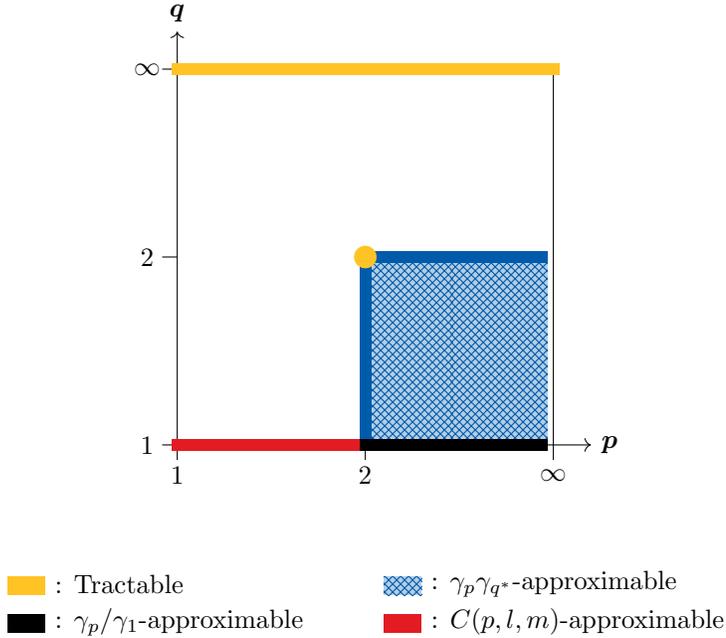
	
	\section{Preliminaries}
	\label{sec:prelims}
	\subsection{Basic Notation}
	\label{sec:basic_notation}
	A letter with an arrow $\v{v}$ represents a vector in $\reals^n$, matrices $\mat{M}$ in $\reals^{n \times m}$ are denoted by bold, underlined letters, whereas vectors in generic vector spaces (i.e., not necessarily $\reals^n$ or $\reals^{n\times m}$) are written using simple letters, e.g., $v$ or $w$. 
	The vectors $\v{e}_i \in \reals^n$ for $i=1,...,n$ are the canonical basis vectors of $\reals^n$. For $\v{v}\in \reals^n$, $v_i$ for $i=1,...,n$ are the coordinates of $\v{v}$. Similarly, for a matrix $\mat{M}$, the coordinate in the $i$-th row and $j$-th column is $M_{ij}$. The notation $\mat{M} \succeq 0$ means $\mat{M}$ is positive semidefinite, $\mat{M}^+$ is the Moore-Penrose pseudoinverse of $\mat{M}$, and $\trace(\mat{M})$ is the trace of $\mat{M}$. For simplicity, we call a matrix injective/surjective/bijective if the corresponding linear map is injective/surjective/bijective.
	The $n$-dimensional identity matrix is $\mat{I}_n$, whereas $\mat{0}_{n\times m}$ is the $n\times m$-matrix filled with zeros, $\v{1}_n$ the $n$-dimensional vector with only ones and $\v{0}_n$ the $n$-dimensional vector with only zeros, and we drop the indices unless there is a risk of confusion.
	For $\v{v}\in\reals^n$, $\Diag(\v{v})$ is the diagonal matrix with the entries of $\v{v}$ on its diagonal. On the other hand, for a matrix $\mat{M}$, $\diag(\mat{M})$ is the vector corresponding to the diagonal entries of $\mat{M}$. For a linear map $L : \Vset \rightarrow \Wset$, the kernel of $L$ is $\ker(L)$ and its image is $\Image(L)$.
	
	For $\v{v}\in\reals^n$ and $p\in[1,\infty)$, $\norm{\v{v}}_p := \sqrt[p]{|v_1|^p+...+|v_n|^p}$ is the vector $\bm{p}$-norm of $\v{v}$, and for $p=\infty$ we set $\norm{\v{v}}_{\infty} := \max_i |v_i|$. The (closed) unit ball of the $p$-norm is $\ball_p$. More generally, for a normed space $(\Vset, \norm_{\Vset})$, we denote its unit ball by $\ball_{\Vset}$. For $p\in[1,\infty]$, $p^*$ is the Hölder conjugate of $p$, defined through $\frac{1}{p}+\frac{1}{p^*} = 1$.
	If $\v{v},\v{w}\in\reals^n$, we denote the (Euclidean) inner product as $\iprod{\v{v}}{\v{w}}_{\reals^n} = \v{v}^\top\v{w} = v_1w_1+...+v_nw_n$.
	
	For functions $f : X \rightarrow Y$ and $g : Y \rightarrow Z$, the function $g \circ f$ is the composition $x \mapsto g(f(x))$.
	For a random variable $\xrand\in\Xset$ on some measurable space $\Xset$, and a function $f:\Xset \rightarrow \reals$, the expectation value of $f(\xrand)$ with respect to $\xrand$ is $\mean_\xrand[f(\xrand)]$. The notation $\grand \sim \mathcal{N}(\mu, \sigma)$ means $\grand$ is a Gaussian random variable with mean $\mu$ and standard deviation $\sigma$.
	
	\subsection{Extended Seminorms}
	Seminorms are well documented in the mathematical literature (see, for example, \cite[p. 200]{Kubrusly2011}). For certain arguments, we require a similar generalization to that of \cite{Beer2015} for norms:
	\begin{definition}[Extended Seminorms and Norms]
		Let $\Vset$ be a real vector space. A function $\eta : \Vset \rightarrow [0,\infty]$ is an \emph{extended seminorm}, if $\eta(0) = 0$ and
		\begin{itemize}
			\item (Triangle inequality) $$\eta(v+w) \leq \eta(v) + \eta(w), \quad \forall v,w\in \Vset.$$
			\item (Absolute homogeneity) $$\eta(cv) = |c|\eta(v), \quad \forall c\in\reals\backslash\{0\}, v\in\Vset.$$ 
		\end{itemize}
		If, in addition, $\eta$ is positive definite, i.e.,
		$$\forall v \in \Vset, \quad\eta(v) = 0 \;\Rightarrow \; v = 0,$$
		then $\eta$ is an \emph{extended norm}.
	\end{definition}
	The major difference with respect to (semi)norms is that extended (semi)norms may admit a value of $\infty$, which requires adapting the absolute homogeneity assumption to avoid issues when multiplying $\infty$ with $0$.
	
	\subsection{Duality}
	\label{sec:duality}
	Let $(\Vset, \norm_{\Vset})$ be a (real) normed vector space. We denote by $\Vset^*$ its dual vector space, i.e., $\Vset^* = \Set{f: \Vset \rightarrow \reals}{f \text{ is linear and continuous}}$, endowed with the \emph{dual norm} $\norm^*$ (see \cite[Chapter 1.3]{Buehler2018}), defined for $f\in \Vset^*$ as
	\begin{equation}
		\label{eq:dual_norm_def}
		\norm{f}^* = \sup_{\norm{x} \leq 1} f(x).
	\end{equation}
	We define the bidual $\Vset^{**}$ as the dual of $\Vset^*$, endowed with the norm $\norm^{**}$ which is the dual of $\norm^*$. The space $\Vset$ is said to be \emph{reflexive} if the \emph{evaluation map} $\iota_{\Vset}$ is an isometric isomorphism (see \cite[Chapter 2.4]{Buehler2018}), where
	\begin{equation}
		\begin{split}
			\iota_{\Vset} :\; &\Vset \rightarrow \Vset^{**}\\
			&v \mapsto (f \mapsto f(v))
		\end{split}
	\end{equation} 
	
	According to the Riesz representation theorem, if $\Hset$ is a Hilbert space (e.g., $\Hset = \reals^n$) with inner product $\iprod_{\Hset}$, the map $R_{\Hset} : \Hset \rightarrow \Hset^*$ sending $y\in\Hset$ to $(x \mapsto \iprod{y}{x}_{\Hset}) \in \Hset^*$ is an isometric isomorphism. Consequently, we may see certain norms on $\Hset^*$ as norms on $\Hset$. Specifically, let $\norm$ be a norm on $\Hset$, which is not necessarily the norm $\norm_{\Hset}$ induced by the inner product, but is equivalent to it (so that both norms induce the same topology). Then the dual norm $\norm^*$ is well-defined on $\Hset^*$, and we can define the corresponding norm $x \mapsto \norm{R_{\Hset}(x)}^*$ on $\Hset$, which for simplicity we also denote by $\norm{x}^*$ unless there is a risk of confusion.
	For example, $\|\v{x}\|_{p}^* = \|\v{x}\|_{p^*}$ for the vector $p$-norms on $\reals^n$ according to \cite[Chapter A.1.6]{Boyd2004}. For a norm $\norm$ on the space of matrices $\reals^{n\times m}$, the dual norm (see \cite[Chapter A.1.6]{Boyd2004}) is
	\begin{equation}
		\label{eq:matrix_norm_dual_def}
		\norm{\mat{Y}}^* = \sup_{\norm{\mat{X}}\leq 1} \trace(\mat{X}^\top \mat{Y}).
	\end{equation}
	\subsection{Entry-Wise Matrix Norms}
	Based on the vector $p$-norms, we define the $L_{p,q}$- and $L_{p,q}^\top$-norms for matrices:
	\begin{definition}[$L_{p,q}$- and $L_{p,q}^\top$-norms]
		\label{def:matrix_L_p_q_norms}
		For $p,q\in[1,\infty]$ and a matrix $\mat{A}\in\reals^{n\times m}$ with columns $\v{a}_1,\hdots,\v{a}_m$, the $L_{p,q}$-norm is
		\begin{equation}
			\norm{\mat{A}}_{L_{p,q}} = \norm{ \begin{pmatrix}\norm{\v{a}_1}_p & \cdots & \norm{\v{a}_m}_p\end{pmatrix}^\top }_q.
		\end{equation}
		For $p, q \in [1, \infty)$, this equates to
		\begin{equation}
			\label{eq:L_pq_def}
			\norm{\mat{A}}_{L_{p,q}} = \left(\sum_{j=1}^m \left( \sum_{i=1}^n |A_{ij}|^p\right)^{q/p}\right)^{1/q}.
		\end{equation}
		The transposed $L_{p,q}$-norms, or $L_{p,q}^\top$-norms, are defined as
		\begin{equation}
			\norm{\mat{A}}_{L_{p,q}^\top} = \nnorm{\mat{A}^\top}_{L_{p,q}}.
		\end{equation}
	\end{definition}
	The dual of the $L_{p,q}$- and $L_{p,q}^\top$-norms can be deduced in the same way as for the vector $p$-norms:
	\begin{lemma}[Dual of $L_{p,q}$-norms]
		\label{lmm:dual_L_p_q}
		Let $\mat{X}\in\reals^{n\times m}$. For $p, q\in[1,\infty]$,
		\begin{equation}
			\begin{split}
				&\|\mat{X}\|_{L_{p,q}}^* = \|\mat{X}\|_{L_{p^*,q^*}},\\
				&\|\mat{X}\|_{L_{p,q}^\top}^* = \|\mat{X}\|_{L_{p^*,q^*}^\top}.
			\end{split}
		\end{equation}
	\end{lemma}
	\begin{proof}
		For matrices $\mat{X}, \mat{Y} \in \reals^{n\times m}$, we denote their columns by $\v{x}_i$ and $\v{y}_i$, respectively. Then
		\begin{equation*}
			\|\mat{X}\|_{L_{p,q}}^*
			\overset{\eqref{eq:matrix_norm_dual_def}}{=} \max_{\|\mat{Y}\|_{L_{p,q}}\leq 1} \trace(\mat{X}^\top\mat{Y})
			=\max_{\|\v{r}\|_q\leq 1} \;\, \max_{\|\v{y}_i\|_p\leq |r_i|} \;\, \sum_{i=1}^m\v{x}_i^\top\v{y}_i.
		\end{equation*}
		Each summand depends on exactly one $\v{y}_i$, so the sum commutes with the maximum over $\v{y}_i$, which yields:
		\begin{align*}
			\|\mat{X}\|_{L_{p,q}}^*
			&=\max_{\|\v{r}\|_q\leq 1} \sum_{i=1}^m\max_{\|\v{y}_i\|_p\leq |r_i|}\v{x}_i^\top\v{y}_i && \text{(swap $\max$ and $\sum$)}\\
			&=\max_{\|\v{r}\|_q\leq 1} \sum_{i=1}^m|r_i|\max_{\|\v{z}_i\|_p\leq 1}\v{x}_i^\top\v{z}_i && \text{(with the substitution $|r_i|\v{z}_i= \v{y}_i$)}\\
			&=\max_{\|\v{r}\|_q\leq 1} \sum_{i=1}^m|r_i|\norm{\v{x}_i}_{p^*} && \text{(by duality, $\max_{\|\v{z}_i\|_p\leq 1}\v{x}_i^\top\v{z}_i=\norm{\v{x}_i}_{p^*}$)}\\
			&=\max_{\|\v{r}\|_q\leq 1} \sum_{i=1}^mr_i\|\v{x}_i\|_{p^*}&& \text{(by symmetry of $\|\v{r}\|_q\leq1$)}\\
			&=\max_{\|\v{r}\|_q\leq 1} \v{r}^\top\v{v}&& \text{(with $\v{v} := \left(\|\v{x}_1\|_{p^*},\cdots,\|\v{x}_m\|_{p^*}\right)^\top$)}\\
			&=\norm{\v{v}}_{q^*}&& \text{(by duality, $\max_{\|\v{r}\|_q\leq 1}\v{v}^\top\v{r}=\norm{\v{v}}_{q^*}$)}\\
			&=\|\left(\|\v{x}_1\|_{p^*},\cdots,\|\v{x}_m\|_{p^*}\right)^\top\|_{q^*}&&\\
			&=\|\mat{X}\|_{L_{p^*,q^*}}&&
		\end{align*}
		The proof of $\|\mat{X}\|_{L_{p,q}^\top}^* = \|\mat{X}\|_{L_{p^*,q^*}^\top}$ is similar.
	\end{proof}
	\subsection{Banach Spaces and the Banach Space Adjoint}
	We recall that a normed space $(\Vset, \norm_{\Vset})$ is a \emph{Banach space} if $\Vset$ is complete, i.e., every Cauchy sequence in $\Vset$ converges to some point in $\Vset$. For more information on Banach spaces, we refer to \cite[Chapter III]{Reed1981}. A linear operator $L : \Vset \rightarrow \Wset$ between Banach spaces is \emph{bounded} if $\norm{L} < \infty$, where $\norm{L}$ is the operator norm from \eqref{eq:basic_problem}. For more information on bounded operators, including a proof that they are continuous, we refer to \cite[Chapter 4.4]{Kubrusly2011}. Any linear operator between Banach spaces gives rise to an adjoint operator:
	\begin{definition}[Banach Space Adjoint]
		\label{def:adjoint}
		Let $L : \Vset \rightarrow \Wset$ be a bounded linear operator between Banach spaces. The \emph{Banach space adjoint} is the operator $L^* : \Wset^* \rightarrow \Vset^*$ defined for $f \in \Wset^*$ and $v\in \Vset$ through
		\begin{equation}
			L^*(f)(v) = f(L(v)).
		\end{equation}
	\end{definition}
	The Banach space adjoint can be thought of as a generalization of the Hilbert space adjoint, which corresponds to the transpose operator on $\reals^n$ or $\reals^{n\times m}$. For more information on adjoints, we refer to \cite[Chapter VI.2]{Reed1981}.
	\begin{lemma}
		\label{lmm:adjoint_properties}
		Let $L : \Vset \rightarrow \Wset$ be a bounded linear operator between Banach spaces. Then $L^* : \Wset^* \rightarrow \Vset^*$ is a bounded linear operator. Moreover, if $L$ is injective and has closed image, then $L^*$ is surjective. If $L$ is surjective, then $L^*$ is injective.
	\end{lemma}
	\begin{proof}
		The fact that $L^*$ is linear is obvious and, since $L$ is bounded, for any $f \in \Wset^*$
		\begin{equation*}
			\norm{L^*(f)}_{\Vset^*} = \sup_{\norm{v}_{\Vset}\leq 1} f(L(v)) \leq \sup_{\norm{v}_{\Vset}\leq 1} \norm{f}_{\Wset^*} \cdot \norm{L(v)}_{\Wset} = \norm{f}_{\Wset^*}\cdot \norm{L}.
		\end{equation*}
		Thus, according to \cite[Chapter 4.4]{Kubrusly2011}, $L^*$ is bounded.
		
		Assume now $L$ is injective and has closed image. Since $L$ is a bounded linear operator, according to the Closed Graph Theorem (i.e., \cite[Theorem 2.2.13]{Buehler2018}) it has closed graph in $\Vset \times \Wset$ with respect to the product topology. By assumption, $L$ additionally has closed image, thus we conclude using the Closed Image Theorem (i.e., \cite[Theorem 6.2.3 (ii)$\Rightarrow$(iv)]{Buehler2018}) that $\Image(L^*) = \ker(L)^{\perp}$, where $\ker(L)^{\perp}$ is the annihilator of $\ker(L)$, see \cite[Definition 2.3.21]{Buehler2018}. Since $L$ is injective, there holds $\ker(L) = 0$, which entails $\ker(L)^{\perp} = \Vset^*$ so that $\Image(L^*) = \Vset^*$, i.e., $L^*$ is surjective.
		
		Finally, assume $L$ is surjective, and suppose $f \in \Wset^*$ satisfies $L^*(f) = 0$. By the definition of $L^*$, this would mean $f(L(v)) = 0$ for every $v\in\Vset$. Since $L$ is surjective, this is equivalent to $f(w) = 0$ for every $w \in \Wset$, that is, $f = 0$. Thus, $\ker(L^*) = \{0\}$, which proves $L^*$ is injective.
	\end{proof}

	\section{Constrained Norm Optimization}
	\label{sec:push-forward_pull-back}
	We turn to the problem of computing expressions of the form
	\begin{equation}
		\inf_{L(x) = c}\norm{x},
	\end{equation}
	where $\norm$ is a norm, $L$ a linear map, and $c$ a fixed vector in some vector space. Specifically, the goal of this section will be to prove the following generalization of \cite[Eq. (5.12)]{Boyd2004}:
	\begin{proposition}[Constrained Matrix Norm Optimization]
		\label{prop:constrained_norm_optimization}
		Let $\norm$ be a matrix norm on $\reals^{l\times m}$, and let $\mat{M} \in \reals^{n\times l}$ be full rank with $n\leq l$. Moreover, let $\mat{C} \in \reals^{n\times m}$ be a fixed matrix. Then,
		\begin{equation}
			\label{eq:constrained_norm_optimization}
			\min_{\mat{M}\,\mat{X}=\mat{C}} \norm{\mat{X}} = \max_{\norm{\mat{M}^\top\,\mat{Y}}^* \leq 1} \trace(\mat{C}^\top \mat{Y}).
		\end{equation}
	\end{proposition}
	We will formally prove Proposition \ref{prop:constrained_norm_optimization} at the end of Section \ref{sec:dual_pf_pb}. Note that \cite[Eq. (5.12)]{Boyd2004} can be seen as a special case of Proposition \ref{prop:constrained_norm_optimization} for $m=1$.
	In particular, Proposition \ref{prop:constrained_norm_optimization} for vector $p$-norms allows us to construct the dual formulation of the generalized matrix norm problem:
	\begin{corollary}
		\label{cor:matrix_norm_and_optimization}
		Let $\mat{A} \in \reals^{m\times n}$ and $\mat{B} \in \reals^{l\times n}$, assume $\mat{B}$ is injective (i.e., $\mat{B}$ is full rank with $n \leq l$), and let $\norm_{\Vset}$ be a norm on $\reals^{m}$ and $\norm_{\Wset}$ a norm on $\reals^l$. Furthermore, let $\v{c} \in \reals^n$. Then
		\begin{equation}
			\label{eq:matrix_norm_and_optimization}
			\max_{\norm{\mat{B}\v{x}}_{\Wset}\leq 1}\left(\norm{\mat{A}\v{x}}_{\Vset}+\v{c}^\top\v{x}\right) = \max_{\norm{\v{\alpha}}_{\Vset}^*\leq 1} \;\, \min_{\mat{B}^\top\v{\beta} = \mat{A}^\top\v{\alpha} + \v{c}} \;\, \nnorm{\v{\beta}}_{\Wset}^*.
		\end{equation}
		In particular, for $\v{c} = \v{0}$,
		\begin{equation}
			\label{eq:matrix_norm_and_optimization_c=0}
			\max_{\norm{\mat{B}\v{x}}_{\Wset}\leq 1}\norm{\mat{A}\v{x}}_{\Vset} = \max_{\norm{\v{\alpha}}_{\Vset}^*\leq 1} \;\, \min_{\mat{B}^\top\v{\beta} = \mat{A}^\top\v{\alpha}} \;\, \nnorm{\v{\beta}}_{\Wset}^*.
		\end{equation}
	\end{corollary}
	\begin{proof}
		By duality,
		\begin{align*}
			\max_{\norm{\mat{B}\v{x}}_{\Wset}\leq 1}\left(\norm{\mat{A}\v{x}}_{\Vset}+\v{c}^\top\v{x}\right)
			&= \max_{\norm{\mat{B}\v{x}}_{\Wset}\leq 1} \;\, \max_{\norm{\v{\alpha}}_{\Vset}^*\leq 1} \;\, \left(\v{\alpha}^\top\mat{A}\v{x}+\v{c}^\top\v{x}\right)\\
			&= \max_{\norm{\v{\alpha}}_{\Vset}^* \leq 1} \;\, \max_{\norm{\mat{B}\v{x}}_{\Wset}\leq 1} \;\, \left(\v{c}^\top + \v{\alpha}^\top\mat{A}\right)\v{x}.
		\end{align*}
		If $\mat{B}$ is injective, $\mat{B}^\top$ is surjective, therefore Proposition \ref{prop:constrained_norm_optimization} yields \eqref{eq:matrix_norm_and_optimization}.
	\end{proof}
	\begin{remark}
		An important consequence of Corollary \ref{cor:matrix_norm_and_optimization} is that the generalized matrix norm can be related to the containment problem, which is the problem of checking whether a set $\widecheck{S}$ is contained in a set $\widehat{S}$. Specifically, if $\widecheck{S} = \Set{\v{c} + \mat{A}^\top \v{\alpha}}{\norm{\v{\alpha}}_{\Vset}^* \leq 1}$ and $\widehat{S} = \Set{\v{d} + \mat{B}^\top \v{\beta}}{\nnorm{\v{\beta}}_{\Wset}^* \leq 1}$,
		\begin{equation}
			r(\widecheck{S}, \widehat{S}) = \max_{\norm{\v{\alpha}}_{\Vset}^*\leq 1} \quad \min_{\mat{B}^\top\v{\beta} + \v{d} = \mat{A}^\top\v{\alpha} + \v{c}} \quad \nnorm{\v{\beta}}_{\Wset}^*
		\end{equation}
		is the smallest scalar such that $\widecheck{S} \subseteq r(\widecheck{S}, \widehat{S})\cdot(\widehat{S}-\v{d})+\v{d}$. In other words, $r(\widecheck{S}, \widehat{S})$ is the multiplicative factor by which $\widehat{S}$ can be shrinked with respect to its center, so it still contains $\widecheck{S}$. In particular, $\widecheck{S}\subseteq \widehat{S}$ if and only if $r(\widecheck{S}, \widehat{S}) \leq 1$. This is explored in more detail in our companion paper \cite{Kulmburg2024}.
	\end{remark}
	To prove Proposition \ref{prop:constrained_norm_optimization}, the concepts of push-forward and pull-back of seminorms will prove fruitful. We will first discuss the case where $\norm$ is the norm of a general (reflexive) Banach space before considering the special cases of matrix norms.
	\subsection{Push-Forward and Pull-Back Seminorms}
	The push-forward and pull-back of a seminorm can be defined in a manner similar to other branches of mathematics:
	\begin{definition}[Push-forward and Pull-back seminorms]
		\label{def:push-foward_pull-back}
		Let $(\Vset,\norm_\Vset)$ and $(\Wset,\norm_\Wset)$ be normed vector spaces, and $L: \Vset \rightarrow \Wset$ a linear operator. The \emph{push-forward} seminorm on $\Wset$ induced by $L$ is the extended seminorm
		\begin{equation}
			\label{eq:push-forward}
			\norm{w}_{L\uparrow\Vset} := \inf_{L(v) = w} \norm{v}_\Vset.
		\end{equation}
		We denote the unit ball of $\norm_{L\uparrow\Vset}$ as $\ball_{L\uparrow\Vset}$, and the space $\Wset$ endowed with that extended seminorm as $L\uparrow \Vset$.
		
		On the other hand, the \emph{pull-back} seminorm on $\Vset$ induced by $L$ is the seminorm
		\begin{equation}
			\label{eq:pull-back}
			\norm{v}_{L\downarrow\Wset} := \norm{L(v)}_\Wset.
		\end{equation}
		We denote the unit ball of $\norm_{L\downarrow\Wset}$ as $\ball_{L\downarrow\Wset}$, and the space $\Vset$ endowed with that seminorm as $L\downarrow \Wset$.
	\end{definition}
	\begin{example}[Ellipsotopes]
		\label{rmk:ellipsotopes}
		Let $\Vset,\Wset$ be vector spaces with a surjective linear operator $L: \Vset\rightarrow \Wset$ and a norm $\norm_\Vset$ on $\Vset$. Then the unit ball $\ball_{L\uparrow\Vset}$ of the push-forward can be formulated as
		\begin{equation}
			\ball_{L\uparrow\Vset} = \Set{L(v)}{\norm{v}_\Vset\leq 1} = L(\ball_{\Vset}).
		\end{equation}
		In other words, $\ball_{L\uparrow\Vset}$ is the image of the unit ball of $\norm_{\Vset}$ under the linear map $L$. If we add a vector $c$ to this set, i.e., if we consider the affine image $c+ L(\ball_{\Vset})$, we can construct many familiar set representations. For example, zonotopes are just the affine image of a hypercube (i.e., the unit ball of the $\infty$-norm). In contrast, ellipsoids can be represented as the affine image of an Euclidean ball (i.e., the unit ball of the Euclidean norm). Any symmetric polytope can be seen as the affine image of a cross-polytope (i.e., the unit ball of the $1$-norm). More generally, if the underlying norm is a vector $p$-norm, such sets are called \emph{ellipsotopes} in \cite{Kousik2022}.
	\end{example}
	We now show that both the push-forward and the pull-back are (extended) seminorms:
	\begin{lemma}
		\label{lmm:pf_pb_are_norms}
		For normed vector spaces $(\Vset, \norm_\Vset)$ and $(\Wset, \norm_\Wset)$ and a linear map $L:\Vset \rightarrow \Wset$, the push-forward is an extended seminorm, and the pull-back is a seminorm.
		Moreover, if $L$ is continuous, then the infimum in \eqref{eq:push-forward} can be replaced by a minimum if \eqref{eq:push-forward} is bounded. In this case, the push-forward $\norm_{L\uparrow\Vset}$ is a norm if and only if $L$ is surjective. The pull-back $\norm_{L\downarrow\Wset}$ is a norm if and only if $L$ is injective.
	\end{lemma}
	\begin{proof}
		Both $\norm_{L\uparrow\Vset}$ and $\norm_{L\downarrow\Wset}$ can easily be seen to be absolutely homogeneous. Using the linearity of $L$ and the triangle inequality for $\norm_\Wset$, the triangle inequality for the pull-back follows. For the push-forward, we need more refined arguments: If for $w_1, w_2 \in \Wset$ either $\norm{w_1}_{L\uparrow\Vset} = \infty$ or $\norm{w_2}_{L\uparrow\Vset} = \infty$, or both, then the triangle inequality trivially holds. Therefore, assume without loss of generality that $\norm{w_1}_{L\uparrow\Vset} \neq \infty \neq \norm{w_2}_{L\uparrow\Vset}$. Let $(v_{1,i})_{i\in\naturals}$ and $(v_{2,i})_{i\in\naturals}$ be sequences satisfying $L(v_{1,i}) = w_1$ and $L(v_{2,i}) = w_2$ for all $i\in\naturals$, and such that $\norm{v_{1,i}}_\Vset \rightarrow \norm{w_1}_{L\uparrow\Vset}$ and $\norm{v_{2,i}}_\Vset \rightarrow \norm{w_2}_{L\uparrow\Vset}$. Clearly, $L(v_{1,i}+v_{2,i}) = w_1+w_2$ for all $i\in\naturals$ by linearity of $L$.
		As a consequence, for any $i\in\naturals$,
		\begin{equation*}
			\norm{w_1+w_2}_{L\uparrow\Vset} = \inf_{L(v) = w_1 + w_2} \norm{v}_\Vset \quad
			\leq \norm{v_{1,i}+v_{2,i}}_\Vset 
			\leq \norm{v_{1,i}}_{\Vset}+\norm{v_{2,i}}_\Vset.
		\end{equation*}
		Taking the limit $i\rightarrow \infty$ yields the triangle inequality $\norm{w_1+w_2}_{L\uparrow\Vset} \leq \norm{w_1}_{L\uparrow\Vset}+\norm{w_2}_{L\uparrow\Vset}$ for the push-forward.
		\medskip
		
		If $L$ is continuous, for any $w$, the set $\Set{v\in \Vset}{L(v) = w}$ is closed, since $\{w\}$ is closed. Consequently, any sequence of elements $v_i\in \Set{v\in \Vset}{L(v) = w}$ whose norm would converge to the infimum for $i\rightarrow \infty$ converges inside $\Set{v\in \Vset}{L(v) = w}$, and this limit point is then the minimizer. If there are no solutions such that $L(v) = w$, we take the infimum over an empty set, which is equal to $\infty$ by convention. Since $\infty$ can be a possible value for an extended seminorm, even in that case the minimum is reached.
		
		We now examine under which conditions the push-forward and the pull-back are norms, beginning with the latter: since the pull-back is a seminorm, it is a norm if and only if it is positive definite, i.e., $\norm{v}_{L\downarrow\Wset} = 0 \Rightarrow v = 0$. Since $\norm{v}_{L\downarrow\Wset} = \norm{L(v)}_{\Wset}$, the inequality $\norm{v}_{L\downarrow\Wset} \neq 0$ for $v\neq 0$ occurs if and only if $\text{Ker}(L)=\{0\}$, which is equivalent to $L$ being injective. For the push-forward, note that $\norm_{L\uparrow\Vset}$ is a bounded seminorm (i.e., it does not output $\infty$) if and only if $L(v) = w$ has at least one solution $v$ for any $w$, which is equivalent to $L$ being surjective. Moreover, in the case where $L$ is surjective and continuous, suppose $w\in \Wset$ satisfies $\norm{w}_{L\uparrow\Vset} = 0$. As discussed above, the continuity of $L$ implies that the infimum in \eqref{eq:push-forward} is reached, so there exists $v \in \Vset$ with $L(v) = w$ and $\norm{v}_\Vset = 0$. Since $\norm_\Vset$ is a norm, it is positive definite, hence $v = 0$ must hold, thus $w = L(v) = L(0) = 0$ since $L$ is linear, which proves that $\norm_{L\uparrow\Vset}$ is also positive definite.
	\end{proof}
	
	\subsection{The Dual of the Push-Forward and Pull-Back}
	\label{sec:dual_pf_pb}
	A surprising and novel property of the push-forward and pull-back norms is that they are dual to each other: 
	\begin{theorem}[Duality of the Push-forward and Pull-back]
		\label{thm:push-forward_pull-back_dual_banach}
		Let $(\Vset, \norm_{\Vset})$ and $(\Wset, \norm_{\Wset})$ be Banach spaces, and $L : \Vset \rightarrow \Wset$ a bounded, surjective linear operator with adjoint $L^*$. Then, for any $g\in\Wset^*$,
		\begin{equation}
			\label{eq:push-forward_pull-back_dual_banach}
			\norm{g}_{L\uparrow\Vset}^* = \norm{g}_{L^*\downarrow\Vset^*}.
		\end{equation}
		
		On the other hand, if $\Vset$ and $\Wset$ are reflexive and $M : \Wset^* \rightarrow \Vset^*$ is a bounded, injective linear operator with closed image and adjoint $M^*$, then for any $\psi \in \Wset^{**}$,
		\begin{equation}
			\label{eq:push-forward_pull-back_dual_banach_part2}
			\norm{\psi}_{M\downarrow\Vset^*}^* = \norm{\psi}_{M^*\uparrow\Vset^{**}}.
		\end{equation}
	\end{theorem}
	\begin{proof}
		Suppose $L : \Vset \rightarrow \Wset$ is a bounded surjective linear operator. For any $g\in\Wset^*$,
		\begin{equation*}
			\norm{g}_{L\uparrow\Vset}^*					
			=\sup_{\substack{w\in\Wset\\w\neq0}}\frac{g(w)}{\|w\|_{L\uparrow\Vset}} 									
			=\sup_{\substack{w\in\Wset\\w\neq0}} \;\, \sup_{L(v)=w}\frac{g(w)}{\norm{v}_{\Vset}} 							
			=\sup_{\substack{w\in\Wset\\L(v)=w\\w\neq0}}\frac{g(L(v))}{\norm{v}_{\Vset}} .						
		\end{equation*}
		Since $L$ is surjective, the supremum can be taken over all $v\in \Vset$ with $L(v)\neq 0$, since any $w\in\Wset\backslash\{0\}$ is the image of some $v\in\Vset$ under $L$. Therefore,
		\begin{equation}
			\label{eq:detour}
			\norm{g}_{L\uparrow\Vset}^* = \sup_{\substack{v\in\Vset\\L(v)\neq0}}\frac{g(L(v))}{\norm{v}_{\Vset}} = \sup_{\substack{v\in\Vset\\v\neq0}}\frac{g(L(v))}{\norm{v}_{\Vset}}.
		\end{equation}
		To prove the second equality in \eqref{eq:detour}, note that if $g = 0$ then \eqref{eq:detour} trivially holds. If $g \neq 0$, since $L$ is surjective, there must exist a $v\in\Vset$ with $v\neq 0$ such that $g(L(v)) > 0$, which can happen if and only if $L(v) \neq 0$.
		We conclude
		\begin{equation*}
			\norm{g}_{L\uparrow\Vset}^*
			=\sup_{\norm{v}_{\Vset}\leq 1}g(L(v))
			=\sup_{\norm{v}_{\Vset}\leq 1}(L^*(g))(v)																															
			=\norm{L^*(g)}_{\Vset}^*
			=\norm{g}_{L^*\downarrow\Vset^*},
		\end{equation*}
		where for the second equality we used the definition of the adjoint (see Definition \ref{def:adjoint}).
		
		Suppose now $M : \Wset^* \rightarrow \Vset^*$ is a bounded, injective linear operator with closed image, so that $M^*$ is surjective by Lemma \ref{lmm:adjoint_properties}, and hence $\norm{\psi}_{M^*\uparrow\Vset^{**}} < \infty$ for any $\psi \in \Wset^{**}$ according to Lemma \ref{lmm:pf_pb_are_norms}. From now on, let $\psi \in \Wset^{**}$ be arbitrary but fixed and denote $\theta := \norm{\psi}_{M^*\uparrow\Vset^{**}}$. We will show $\norm{\psi}_{M\downarrow\Vset^*}^* \leq \norm{\psi}_{M^*\uparrow\Vset^{**}} = \theta$ using the Hahn-Banach separation theorem, similarly to \cite[Chapter 5.3]{Boyd2004}. To that end, for some small $\varepsilon$ such that $0 < \varepsilon < \theta$ we define
		\begin{align*}
			S_{\varepsilon} &:= \Set{(0, z)\in \Wset^{**} \times \reals}{0 \leq z \leq \theta - \varepsilon},\\
			T &:= \Set{(\phi, z) \in \Wset^{**} \times \reals}{\exists \varphi \in \Vset^{**}, M^*(\varphi) = \psi + \phi \text{ and } \norm{\varphi}_{\Vset^{**}}\leq z}.
		\end{align*}
		Note that $S_{\varepsilon} = \{0\} \times [0, \theta - \varepsilon]$, so $S_{\varepsilon}$ is convex and compact. On the other hand, since $M^*$ is bounded (and thus continuous) by Lemma \ref{lmm:adjoint_properties}, we can use Lemma \ref{lmm:pf_pb_are_norms} to show
		\begin{equation*}
			T = \Set{(\phi, z) \in \Wset^{**} \times [0,\infty)}{\norm{\psi + \phi}_{M^*\uparrow\Vset^{**}} \leq z}.
		\end{equation*}
		Since $M^*$ is surjective, $\norm_{M^*\uparrow\Vset^{**}}$ is a norm according to Lemma \ref{lmm:pf_pb_are_norms}, so in particular it is a continuous and convex function. Therefore, $T$ is closed and convex. Moreover, $T$ and $S_{\varepsilon}$ are clearly disjoint: For $(\phi, z) \in T$ to be in $S_{\varepsilon}$, there would need to hold $\phi = 0$ and $z \leq \theta - \varepsilon$, so that by definition of $T$ there must exist a $\varphi\in\Vset^{**}$ that satisfies $M^*(\varphi) = \psi$ and $\norm{\varphi}_{\Vset}^{**}\leq z$. But by definition of the push-forward $M^*\uparrow \Vset^{**}$, this means $\theta = \norm{\psi}_{M^*\uparrow \Vset^{**}} \leq z$, implying
		\begin{equation*}
			z \leq \theta - \varepsilon \leq z - \varepsilon,
		\end{equation*}
		which is a contradiction since $\varepsilon > 0$. Finally, $\Wset^{**} \times \reals$ is locally convex because $\Wset^{**}$ is a Banach space (see \cite[Theorem 5.5.2]{Narici2010}), so we can apply the Hahn-Banach separation theorem (i.e., \cite[Theorem 7.8.6(a)]{Narici2010}) on $S_{\varepsilon}, T \subset \Wset^{**} \times \reals$ to find a hyperplane strictly separating $S_{\varepsilon}$ and $T$. This means there exist $G \in (\Wset^{**})^*$ and $c \in \reals$ such that
		\begin{align}
			&\sup_{(\phi,z) \in S_{\varepsilon}} G(\phi) + cz < \inf_{(\phi, z) \in T} G(\phi) + cz \notag\\
			\Leftrightarrow \quad &\sup_{0 \leq z \leq \theta - \varepsilon}c z < \inf_{(\phi, z) \in T} G(\phi) + cz.\label{eq:proof_lmm_duality_hahn_banach}
		\end{align}
		We now show $c>0$: If $c$ were negative, then the left-hand side of \eqref{eq:proof_lmm_duality_hahn_banach} would be finite, but the right-hand side would be $-\infty$, since by definition of $T$ we can choose $z$ arbitrarily large and $\phi = 0$. If $c$ were zero, \eqref{eq:proof_lmm_duality_hahn_banach} would imply $0 < \inf_{(\phi, z) \in T} G(\phi)$. However, this can not hold for $(\phi, z) = (0, \theta) \in T$, so there must hold $c>0$. Consequently, we can divide both sides of \eqref{eq:proof_lmm_duality_hahn_banach} by $c$, to obtain
		\begin{equation}
			\label{eq:proof_conclusion_hahn-banach}
			\theta - \varepsilon < \inf_{(\phi, z) \in T} \frac{1}{c}G(\phi) + z.
		\end{equation}
		Let $\varphi \in \Vset^{**}$ be arbitrary, and choose $\phi = M^*(\varphi) - \psi$ and $z = \norm{\varphi}_{\Vset^{**}}$. Clearly, $(\phi, z) \in T$, so that \eqref{eq:proof_conclusion_hahn-banach} implies
		\begin{equation}
			\label{eq:proof_conclusion_hahn-banach2}
			\theta - \varepsilon < \frac{1}{c}G(M^*(\varphi) - \psi) + \norm{\varphi}_{\Vset^{**}}.
		\end{equation}
		Since $\Wset$ is reflexive, so is $\Wset^*$ (see \cite[Theorem 2.4.4]{Buehler2018}), thus any $\frac{1}{c}G \in (\Wset^{**})^* \cong (\Wset^*)^{**}$ can be written as $\frac{1}{c}G = \iota_{\Wset^*}(g)$ for some $g\in\Wset^*$, where $\iota_{\Wset^*}$ is the evaluation map on $\Wset^*$ (see \cite[Chapter 2.4]{Buehler2018}), so we can rewrite \eqref{eq:proof_conclusion_hahn-banach2} as
		\begin{equation*}
			\theta - \varepsilon < M^*(\varphi)(g) - \psi(g) + \norm{\varphi}_{\Vset^{**}}.
		\end{equation*}
		The $\varphi\in\Vset^{**}$ was arbitrary, so we can take the infimum over all $\varphi$ and obtain
		\begin{align*}
			&\theta - \varepsilon < -\psi(g) + \inf_{\varphi\in \Vset^{**}} \left(M^*(\varphi)(g) + \norm{\varphi}_{\Vset^{**}}\right)\\
			\Leftrightarrow \quad &\theta - \varepsilon < -\psi(g) + \inf_{\varphi\in \Vset^{**}} \left(\varphi(M(g)) + \norm{\varphi}_{\Vset^{**}}\right)\\
			\Leftrightarrow \quad &\theta - \varepsilon < -\psi(g) - \sup_{\varphi'\in \Vset^{**}} \left(\varphi'(M(g)) - \norm{\varphi'}_{\Vset^{**}}\right),
		\end{align*}
		where we used the definition of the adjoint $M^*$ as well as the variable transformation $\varphi' := -\varphi$, together with the symmetry of $\norm_{\Vset^{**}}$. We use the reflexivity of $\Vset$ to find a $v \in \Vset$ such that $\varphi' = \iota_{\Vset}(v)$, where $\iota_{\Vset}$ is the evaluation map on $\Vset$. This $v$ satisfies $\norm{v}_{\Vset} = \norm{\varphi'}_{\Vset^{**}}$, so that we may write
		\begin{equation}
			\label{eq:proof_conclusion_almost_conv_conj}
			\theta - \varepsilon < -\psi(g) - \sup_{v\in \Vset} \left(M(g)(v) - \norm{v}_{\Vset}\right).
		\end{equation}
		Note that the supremum corresponds to the convex conjugate of $\norm_{\Vset}$ (see Appendix \ref{sec:convex_conjugate}), so by Lemma \ref{lmm:convex_conjugate_norm} there must hold $\norm{M(g)}_{\Vset^*} \leq 1$ (otherwise the right-hand side of \eqref{eq:proof_conclusion_almost_conv_conj} is $-\infty$, a contradiction), in which case $\theta - \varepsilon < -\psi(g)$. Taking the supremum over all such $g\in\Wset^*$ yields
		\begin{equation*}
			\theta - \varepsilon < \sup_{\norm{M(g)}_{\Vset^*} \leq 1} - \psi(g) \overset{g':=-g}{=} \sup_{\norm{M(g')}_{\Vset^*} \leq 1} \psi(g') = \norm{\psi}_{M\downarrow\Vset^*}^*.
		\end{equation*}
		Since this holds for any $0 < \varepsilon < \theta$, we can take $\varepsilon \rightarrow 0$ to obtain $\norm{\psi}_{M^*\uparrow\Vset^{**}} = \theta \leq \norm{\psi}_{M\downarrow\Vset^*}^*$. For the converse inequality, we can use \cite[Theorem 1., p. 217]{Luenberger1969} to rewrite the pull-back:
		\begin{equation*}
			\norm{\psi}_{M\downarrow\Vset^*}^* = \sup_{\norm{M(g)}_{\Vset^*} \leq 1} \psi(g)
			= \sup_{g\in\Wset^*} \;\, \inf_{\lambda \geq 0} \left(\psi(g) + \lambda(1-\norm{M(g)}_{\Vset^*})\right).
		\end{equation*}
		For any function $\nu : \Xset \times \Yset \rightarrow \reals$ over any two sets $\Xset$ and $\Yset$ there always holds
		\begin{equation*}
			\sup_{x\in\Xset} \inf_{y\in\Yset} \nu(x,y) \leq \inf_{y\in\Yset} \sup_{x\in\Xset} \nu(x,y).
		\end{equation*}
		Therefore, using the definition of the dual norm,
		\begin{align*}
			\norm{\psi}_{M\downarrow\Vset^*}^* &= \sup_{g\in\Wset^*} \;\, \inf_{\lambda \geq 0} \;\, \inf_{\norm{v}_{\Vset}\leq 1} \left(\psi(g) + \lambda(1-M(g)(v))\right)\\
			&\leq \inf_{\lambda \geq 0} \;\, \inf_{\norm{v}_{\Vset}\leq 1} \big(\lambda + \sup_{g\in\Wset^*} \left(\psi - \lambda\iota_{\Vset}(v)\circ M\right)(g)\big),
		\end{align*}
		where $\iota_{\Vset}$ is again the evaluation map on $\Vset$.
		Clearly,
		\begin{equation*}
			\sup_{g\in\Wset^*} \left(\psi - \lambda\iota_{\Vset}(v)\circ M\right)(g) = \begin{cases}
				0, &\text{ if } \psi - \lambda\iota_{\Vset}(v)\circ M = 0,\\
				\infty, &\text{ otherwise}.
			\end{cases}
		\end{equation*}
		We conclude
		\begin{equation*}
			\norm{\psi}_{M\downarrow\Vset^*}^* \leq \inf_{\substack{\lambda \geq 0\\\norm{v}_{\Vset}\leq 1 \\\lambda\iota_{\Vset}(v)\circ M = \psi}} \lambda \overset{v':=\lambda v}{=} \inf_{M^*(\iota_{\Vset}(v')) = \psi} \norm{v'}_{\Vset}.
		\end{equation*}
		Since $\Vset$ is reflexive, $\iota_{\Vset}$ is an isometry, so that $\norm{\iota_{\Vset}(v')}_{\Vset^{**}} = \norm{v'}_{\Vset}$. If we define $\varphi := \iota_{\Vset}(v')$ we thus get
		\begin{equation*}
			\norm{\psi}_{M\downarrow\Vset^*}^* \leq \inf_{M^*(\varphi) = \psi} \norm{\varphi}_{\Vset^{**}},
		\end{equation*}
		which proves $\norm{\psi}_{M\downarrow\Vset^*}^* \leq \norm{\psi}_{M^*\uparrow\Vset^{**}}$.
	\end{proof}
	An immediate consequence is the following Corollary:
	\begin{corollary}[Constrained Norm Optimization]
		\label{cor:constrained_norm_optimization_banach}
		Let $(\Vset, \norm_{\Vset})$ and\\$(\Wset, \norm_{\Wset})$ be reflexive Banach spaces, $L : \Vset \rightarrow \Wset$ a bounded, surjective linear operator, and $c\in \Wset$ a fixed vector. Then,
		\begin{equation}
			\label{eq:constrained_norm_optimization_banach}
			\min_{L(x)=c} \norm{x}_{\Vset} = \sup_{\norm{L^*(g)}_{\Vset}^* \leq 1} g(c).
		\end{equation}
	\end{corollary}
	\begin{proof}
		Since $L$ is surjective, it trivially has a closed image, thus according to the Closed Image Theorem (i.e., \cite[Theorem 6.2.3 (ii)$\Rightarrow$(vi)]{Buehler2018}) the operator $M:= L^*$ has closed image, and is injective according to Lemma \ref{lmm:adjoint_properties}.
		We can thus use \eqref{eq:push-forward_pull-back_dual_banach_part2} for $M$ and $\psi := \iota_{\Wset}(c)$, where $\iota_{\Wset}$ is the evaluation map on $\Wset$:
		\begin{equation*}
			\min_{L^{**}(\varphi) = \iota_{\Wset}(c)} \norm{\varphi}_{\Vset}^{**} = \norm{\iota_{\Wset}(c)}_{L^{**}\uparrow \Vset^{**}} = \norm{\iota_{\Wset}(c)}^*_{L^*\downarrow \Vset^*} = \sup_{\norm{L^*(g)}_{\Vset^*}\leq 1} g(c).
		\end{equation*}
		Since $\Vset$ is reflexive, $\iota_{\Vset}$ is an isometric isomorphism (see \cite[Chapter 2.4.]{Buehler2018}). This implies
		\begin{equation}
			\label{eq:cno_part1}
			\min_{L^{**}(\iota_{\Vset}(v)) = \iota_{\Wset}(c)} \norm{v} = \min_{L^{**}(\iota_{\Vset}(v)) = \iota_{\Wset}(c)} \norm{\iota_{\Vset}(v)}_{\Vset}^{**} = \min_{L^{**}(\varphi) = \iota_{\Wset}(c)} \norm{\varphi}_{\Vset}^{**},
		\end{equation}
		where the last equality follows from the bijectivity of $\iota_{\Vset}$.
		According to the Hahn-Banach theorem (see \cite[Corollary 2.3.23.]{Buehler2018}), $w_1 = w_2$ for $w_1,w_2\in\Wset$ holds if and only if $g(w_1) = g(w_2)$ for all $g \in \Wset^*$. Therefore, using the definition of the Banach space adjoint and the evaluation map,
		\begin{equation}
			\label{eq:cno_part2}
			\begin{split}
				& \quad L^{**}(\iota_{\Vset}(v)) = \iota_{\Wset}(c),\\
				\Leftrightarrow & \quad \forall g \in \Wset^*, (L^{**}(\iota_{\Vset}(v)))(g) = (\iota_{\Wset}(c))(g),\\
				\Leftrightarrow & \quad \forall g \in \Wset^*, g(L(v)) = g(c),\\
				\Leftrightarrow & \quad L(v) = c.
			\end{split}
		\end{equation}
		Combining \eqref{eq:cno_part1} with \eqref{eq:cno_part2} yields \eqref{eq:constrained_norm_optimization_banach}.
	\end{proof}
	\begin{example}[The Dual of $\norm{\mat{B}\v{v}}_p$ is not a $p^*$-norm]
		\label{ex:dual_B_p_not_p_star}
		Suppose $\mat{B}$ is injective, so that $\norm{\mat{B}\v{v}}_p$ is a norm for $p\in[1,\infty]$. Then Corollary \ref{cor:constrained_norm_optimization_banach} implies that the dual of $\norm{\mat{B}\v{v}}_p$ is the norm
		\begin{equation*}
			\v{w} \mapsto \min_{\mat{B}^\top\v{\alpha} = \v{w}}\norm{\v{\alpha}}_{p^*}.
		\end{equation*}
		For $p=1$, this corresponds to the zonotope-norm introduced in \cite[Definition 4.]{Kulmburg2021}. According to \cite[Corollary 1.]{Kulmburg2021}, its unit circle is a zonotope (a special type of ellipsotope, see Example \ref{rmk:ellipsotopes}) which is, in general, different from a hypercube (which would be the unit ball of the $\infty$-norm). This shows that the dual of $\norm{\mat{B}\v{v}}_{1}$ and the $\infty$-norm are different, and thus that the dual of $\norm{\mat{B}\v{v}}_p$ can not be formulated in terms of the $p^*$-norm alone.
	\end{example}
	We can now prove Proposition \ref{prop:constrained_norm_optimization}, as it is an easy consequence of Corollary \ref{cor:constrained_norm_optimization_banach}:
	\begin{proof}[Proof of Proposition \ref{prop:constrained_norm_optimization}]
		Since $\mat{M}$ is full rank and $n\leq l$, $\mat{M}$ is surjective. Finite-dimensional vector spaces are always reflexive (see \cite[Example 2.4.5.]{Buehler2018}), so \eqref{eq:constrained_norm_optimization} directly follows from \eqref{eq:constrained_norm_optimization_banach} by using the definition of the transpose (also called the Hilbert space adjoint, see \cite{Reed1981}) and \eqref{eq:matrix_norm_dual_def}.
	\end{proof}
	
	\section{Tractable Matrix Norms}
	\label{sec:tractable}
	
	We now explore for which values of $p$ and $q$ the expression $\norm{\mat{A}}_{p\mapsto q; \mat{B}}$ for $\mat{A}\in\reals^{m\times n}, \mat{B}\in\reals^{l\times n}$ can be computed in polynomial time (with respect to $l$, $m$, and $n$).
	Before starting our analysis, we mention that we may, without loss of generality, assume that $\mat{B}$ is injective. Otherwise, $\text{Ker}(\mat{B}) \neq \{\v{0}\}$, and thus one of two things can happen:
	\begin{itemize}
		\item If $\text{Ker}(\mat{B}) \not\subseteq \text{Ker}(\mat{A})$, there exists a vector $\v{v} \in \reals^{n}$ such that $\mat{B}\v{v} = \v{0}$, but $\mat{A}\v{v} \neq \v{0}$. Hence, for any $\lambda > 0$, we have $\norm{\mat{B}\lambda \v{v}}_p = 0\leq 1$, but $\norm{\mat{A}\lambda \v{v}}_q = \lambda \norm{\mat{A}\v{v}}_q = \lambda c$ for some fixed constant $c > 0$. Letting $\lambda \rightarrow \infty$ shows
		\begin{equation*}
			\norm{\mat{A}}_{p\mapsto q; \mat{B}} = \infty,
		\end{equation*}
		i.e., the problem is unbounded.
		\item If $\text{Ker}(\mat{B}) \subseteq \text{Ker}(\mat{A})$, let $k = \text{Rank}(\mat{B})$. Then the singular value decomposition of $\mat{B}$ has the form
		\begin{equation*}
			\mat{B} = \mat{U}\,\mat{\Sigma}\,\mat{V}^\top, \text{ with } \mat{\Sigma} = \begin{bmatrix}\mat{D} & \mat{0}_{k \times (n-k)}\\\mat{0}_{(l-k)\times k} & \mat{0}_{(l-k)\times (n-k)}\end{bmatrix},
		\end{equation*}
		where $\mat{D} \in \reals^{k\times k}$ is diagonal and invertible, and $\mat{U}$ and $\mat{V}$ are orthogonal matrices. Thus, we have $\norm{\mat{A}}_{p\mapsto q; \mat{B}} = \norm{\mat{A}\,\mat{V}}_{p\mapsto q; \mat{U}\,\mat{\Sigma}}$. Since $\text{Ker}(\mat{B}) \subseteq \text{Ker}(\mat{A})$, we have $\text{Ker}(\mat{U}\,\mat{\Sigma}) \subseteq \text{Ker}(\mat{A}\,\mat{V})$, and thus
		\begin{equation*}
			\norm{\mat{A}\,\mat{V}}_{p\mapsto q; \mat{U}\,\mat{\Sigma}} = \norm{\mat{A}\,\mat{V}}_{p\mapsto q; \mat{B}'}, \text{ with } \mat{B}' = \mat{U}\begin{bmatrix}
				\mat{D}\\\mat{0}_{(l-k)\times k}.
			\end{bmatrix}
		\end{equation*}
		where $\mat{B}'$ is now injective.
	\end{itemize}
	We conclude that if $\mat{B}$ is not injective, either the problem is unbounded, or we can reduce the problem to the case where $\mat{B}$ is injective.
	
	\subsection{The Case \texorpdfstring{$p\in[1,\infty]$}{p ∈ [1,oo]}, \texorpdfstring{$q=\infty$}{q = oo}}
	In this case,
	\begin{equation*}
		\norm{\mat{A}}_{p\mapsto \infty;\mat{B}} = \max_i \;\, \max_{\norm{\mat{B}\v{v}}_p \leq 1} \;\, \v{e}_i^\top\mat{A}\v{v}.
	\end{equation*}
	This can be evaluated as the maximum of $m$ convex optimization problems (one for each $i = 1,...,m$), since $\v{e}_i^\top \mat{A}\v{v}$ is linear and the constraint $\norm{\mat{B}\v{v}}_p \leq 1$ is convex. Nevertheless, the nonlinear constraint may lead to minor approximation errors when using, e.g., interior point methods (see \cite[Chapter 11]{Boyd2004}) to solve the convex optimization problem. Even though this error can be made arbitrarily small, we propose another formulation that only has linear constraints, and can thus be solved directly using the methods from \cite[Chapter 10]{Boyd2004}:
	\begin{theorem}
		Let $\mat{A} \in \reals^{m\times n}$ and $\mat{B}\in \reals^{l\times n}$. Then, for any $p\in[1,\infty]$,
		\begin{equation}
			\label{eq:final_q=infty}
			\max_{\norm{\mat{B}\v{v}}_p\leq 1}\norm{\mat{A}\v{v}}_{\infty} = \max_{i} \;\, \min_{\mat{B}^\top\v{\beta} = \mat{A}^\top\v{e}_i} \;\, \nnorm{\v{\beta}}_{p^*},
		\end{equation}
		which can be evaluated in polynomial time with respect to $l$, $m$, and $n$.
	\end{theorem}
	\begin{proof}
		By Corollary \ref{cor:matrix_norm_and_optimization}, if $\mat{B}$ is injective,
		\begin{equation}
			\label{eq:q=infty_reformulation_first}
			\norm{\mat{A}}_{p,\infty;\mat{B}} = \max_{\norm{\v{\alpha}}_1\leq 1} \;\, \min_{\mat{B}^\top\v{\beta} = \mat{A}^\top\v{\alpha}} \;\, \nnorm{\v{\beta}}_{p^*},
		\end{equation}
		and the term
		\begin{equation}
			\label{eq:reformulation_case_q=infty}
			\min_{\mat{B}^\top\v{\beta} = \mat{A}^\top\v{\alpha}}\nnorm{\v{\beta}}_{p^*}
		\end{equation}
		is a norm with respect to $\v{\alpha}$, which means in particular that it is convex. By the Bauer maximum principle, the maximum in \eqref{eq:q=infty_reformulation_first} over $\v{\alpha}$ is attained at $\v{\alpha} = \pm \v{e}_i$, for some $i=1,...,m$. Furthermore, the term \eqref{eq:reformulation_case_q=infty} can be evaluated in polynomial time using, e.g., interior point methods. Consequently, it suffices to compute \eqref{eq:reformulation_case_q=infty} for the $2m$ values $\v{\alpha} = \pm \v{e}_i$ (in fact, $\v{\alpha} = \v{e}_i$ is sufficient, as \eqref{eq:reformulation_case_q=infty} is a norm, and is thus symmetric), which can be done in polynomial time with respect to $l$, $m$, and $n$. Moreover, for a function $f: D \rightarrow \reals$, we use the convention $\min_{x \in D} f(x) = \infty$ if $D = \emptyset$. This implies that \eqref{eq:final_q=infty} still holds even when $\mat{B}$ is not injective.
	\end{proof}
	\subsection{The Case \texorpdfstring{$p = q = 2$}{p = q = 2}}
	In general, when $p=2$, the generalized matrix norm problem can be simplified:
	\begin{lemma}
		\label{lmm:p=2}
		Let $\mat{A} \in \reals^{m\times n}$ and $\mat{B}\in \reals^{l\times n}$. Then, if $\mat{B}$ is injective, or more generally if $\text{Ker}(\mat{B}) \subseteq \text{Ker}(\mat{A})$, for any $q\in[1,\infty]$ 
		\begin{equation}
			\label{eq:p=2}
			\norm{\mat{A}}_{2\mapsto q;\mat{B}} = \norm{\mat{A}\,\mat{B}^+}_{2\mapsto q},
		\end{equation}
		where $\mat{B}^+$ denotes the Moore-Penrose pseudoinverse of $\mat{B}$.
	\end{lemma}
	\begin{proof}
		We will only cover the case where $\mat{B}$ is injective. The general case is similar. Let $\mat{B} = \mat{U}\,\mat{\Sigma}\,\mat{V}^{\top}$ be a singular value decomposition of $\mat{B}$. Since the 2-norm is invariant under orthogonal transformations,
		\begin{equation*}
			\norm{\mat{A}}_{2\mapsto q;\mat{B}} = \max_{\norm{\mat{B}\v{x}}_2\leq1}\norm{\mat{A}\v{x}}_q
			= \max_{\norm{\mat{\Sigma}\v{y}}_2\leq1}\norm{\mat{A}\,\mat{V}\v{y}}_q,
		\end{equation*}
		where we used the variable transformation $\v{y} = \mat{V}^\top\v{x}$. Since $\mat{B}$ is injective, $\mat{\Sigma}$ has the form
		\begin{equation*}
			\mat{\Sigma}_{\mat{B}} = \begin{pmatrix} \mat{D}\\\mat{0}_{(l-n)\times n}\end{pmatrix}
		\end{equation*}
		for some diagonal, invertible matrix $\mat{D} \in \reals^{n\times n}$. The constraint $\nnorm{\mat{\Sigma}\v{y}}_2\leq 1$ simplifies to $\norm{\mat{D}\v{y}}_2\leq 1$, and using the variable transformation $\v{z} = \mat{D}\v{y}$,
		\begin{equation*}
			\norm{\mat{A}}_{2\mapsto q;\mat{B}} = \max_{\norm{\v{z}}_2\leq 1}\norm{\mat{A}\,\mat{V}\,\mat{D}^{-1}\v{z}}_q
			=\max_{\norm{\v{w}}_2\leq 1}\norm{\mat{A}\,\mat{V}\begin{pmatrix} \mat{D}^{-1} & \mat{0}_{n\times (l-n)}\end{pmatrix}\mat{U}^\top\v{w}}_{q},
		\end{equation*}
		where we have extended the vector $\v{z}\in\reals^{n}$ to a vector $\v{z}\,'\in\reals^{l}$, and used the variable transformation $\v{w} = \mat{U}\v{z}\,'$. According to \cite[p. 207, Corollary 1]{BenIsrael2003}, $\mat{B}^+$ coincides with $\mat{V}\begin{pmatrix} \mat{D}^{-1} & \mat{0}_{n\times (l-n)}\end{pmatrix}\mat{U}^\top$, which yields \eqref{eq:p=2}.
	\end{proof}
	We immediately conclude the following result for $p=2$ and $q=2$:
	\begin{theorem}
		Let $\mat{A} \in \reals^{m\times n}$ and $\mat{B}\in \reals^{l\times n}$. If $\mat{B}$ is injective, or more generally if $\text{Ker}(\mat{B}) \subseteq \text{Ker}(\mat{A})$,
		\begin{equation}
			\label{eq:thm_p_2_q_2}
			\max_{\norm{\mat{B}\v{v}}_2\leq 1}\norm{\mat{A}\v{v}}_2 = \norm{\mat{A}\,\mat{B}^+}_{2\mapsto 2},
		\end{equation}
		which can be evaluated in polynomial time with respect to $l$, $m$, and $n$.
		Instead, if $\text{Ker}(\mat{B}) \not\subseteq \text{Ker}(\mat{A})$,
		\begin{equation}
			\label{eq:thm_p_2_q_2_unbounded}
			\max_{\norm{\mat{B}\v{v}}_2\leq 1}\norm{\mat{A}\v{v}}_2 = \infty.
		\end{equation}
	\end{theorem}
	\begin{proof}
		Of course, \eqref{eq:thm_p_2_q_2} follows directly from Lemma \ref{lmm:p=2}, and \eqref{eq:thm_p_2_q_2_unbounded} follows from our discussion at the beginning of Section \ref{sec:tractable}. The polynomial runtime of \eqref{eq:thm_p_2_q_2} follows from the fact that $\mat{B}^+$ can be computed in polynomial time using a full rank factorization of $\mat{B}$ (see \cite[Theorem 5, p. 48]{BenIsrael2003}) and because the $2\mapsto 2$-norm is tractable (see \cite[Example 5.6.6]{Horn2012}).
	\end{proof}
	
	\section{Approximable Matrix Norms}
	\label{sec:approximable}
	We turn towards cases where $\norm{\mat{A}}_{p\mapsto q}$ can not necessarily be computed exactly in polynomial time but where approximations exist. We will present two main methods, one for $1 \leq q \leq 2 \leq p < \infty$, and the other for $q = 1$ and $p\in [1,2)$. As explained in the previous section, we may assume without loss of generality that $\mat{B}$ is injective.
	
	\subsection{The Case \texorpdfstring{$1 < q \leq 2 \leq p < \infty$}{1 < q < 2 < p < oo}}
	For this range of values for $p$ and $q$, \cite[Theorem 1.4.]{Bhattiprolu2023} showed that computing $\norm{\mat{A}}_{p\mapsto q}$ is $\APX$-hard, so it follows that evaluating $\norm{\mat{A}}_{p\mapsto q;\mat{B}}$ is also $\APX$-hard. Consequently, the best we can hope for is to find an approximation. We need several technical lemmas to construct such an algorithm:
	\begin{lemma}
		\label{lmm:E_ugamma}
		Let $\v{\grand}$ denote a standard Gaussian random vector in $\reals^k$, $\v{u} \in \reals^k$ an arbitrary but fixed vector, and $s \in [0,\infty)$. Then
		\begin{equation}
			\mean_{\v{\grand}}\left[\absiprod{\v{u}}{\v{\grand}}^s\right] = \gamma_s^s\norm{\v{u}}_2^s,
		\end{equation}
		where $\gamma_s$ denotes the $s$-th root of the $s$-th moment of a standard Gaussian random variable, i.e., $\gamma_s = \left(\frac{2^{s/2}}{\sqrt{\pi}}\Gamma\left(\frac{s+1}{2}\right)\right)^{1/s}$, where $\Gamma$ is the Euler Gamma function.
	\end{lemma}
	\begin{proof}
		Since $\forall i, \grand_i \in \mathcal{N}(0, 1)$, we have $u_i\grand_i \in \mathcal{N}(0, u_i)$, which implies
		\begin{equation*}
			u_1\grand_1 + ... + u_k\grand_k \in \mathcal{N}(0, \sqrt{u_1^2+...+u_k^2}).
		\end{equation*}
		Therefore, $\grand':= \iprod{\v{u}}{\v{\grand}}$ is a Gaussian random vector with mean 0 and standard deviation $\norm{\v{u}}_2$, so its $s$-th moment is
		\begin{equation}
			\mean_{\v{\grand}}\left[\absiprod{\v{u}}{\v{\grand}}^s\right] = \mean_{\grand'}\left[|\grand'|^s\right] = \gamma_s^s\norm{\v{u}}_2^s,
		\end{equation}
		where for the last equality we used \cite[Eq.~(17)]{Winkelbauer2014}.
	\end{proof}
	We will need a way to transform a maximization problem into the computation of certain expectation values. This can be achieved by using the result from \cite{Bhattiprolu2018}:
	\begin{lemma}
		\label{lmm:vijay}
		Let $\Omega$ be a sample space, $\xrand : \Omega \rightarrow \Xset$ a random variable in some measurable space $\Xset$, and $f_1 : \Xset \rightarrow \reals$, $f_2 : \Xset \rightarrow [0,\infty)$ two functions. If $D:=f_2^{-1}(\{0\})$ has measure zero and $\Xset \backslash D \neq \emptyset$,
		\begin{equation}
			\label{eq:vijay}
			\sup_{y\in\Xset\backslash D}\frac{f_1(y)}{f_2(y)} \geq \frac{\mean_{\xrand}\left[f_1(\xrand)\right]}{\mean_{\xrand}\left[f_2(\xrand)\right]}.
		\end{equation}
	\end{lemma}
	\begin{proof}
		The proof is almost identical to the argument in \cite[p. 10]{Bhattiprolu2018}: Let $\xrand' : \Omega \rightarrow \Xset\backslash D$ be defined through $\xrand'(\omega) = \xrand(\omega)$ if $\xrand(\omega) \not\in D$, and $\xrand'(\omega) = z$ otherwise, where $z$ is some arbitrary but fixed element in $\Xset\backslash D$. Since $D$ has measure zero, for any function $f : \Xset \rightarrow \reals$ there still holds $\mean_{\xrand'}\left[f(\xrand')\right] = \mean_{\xrand}\left[f(\xrand)\right]$.
		Let $\lambda = \mean_{\xrand}\left[f_1(\xrand)\right]/\mean_{\xrand}\left[f_2(\xrand)\right]$, then
		\begin{equation*}
			\sup_{y\in\Xset\backslash D}\left(f_1(y) - \lambda f_2(y)\right) \geq \mean_{\xrand'}\left[f_1(\xrand') - \lambda f_2(\xrand')\right] = \mean_{\xrand}\left[f_1(\xrand) - \lambda f_2(\xrand)\right] = 0,
		\end{equation*}
		which directly entails \eqref{eq:vijay}.
	\end{proof}
	Finally, we need one last lemma concerning the dual of certain semidefinite optimization problems:
	\begin{lemma}
		\label{lmm:semi-definite_dual}
		Let $\mat{M} \in \reals^{n\times n}$, $\mat{Q} \in \reals^{n\times m}$, and $\mat{R} \in \reals^{n\times l}$. Furthermore, let $\v{s} \in \reals^m$ and $\v{t} \in \reals^l$ be nonnegative vectors, that is, $s_i\geq 0$ and $t_j \geq 0$ for all $i,j$. Then
		\begin{equation}
			\max_{\substack{\mat{Z}\succeq 0\\\diag(\mat{R}^\top\mat{Z}\,\mat{R}) = \v{t}\\\diag(\mat{Q}^\top\mat{Z}\,\mat{Q}) = \v{s}}} \trace(\mat{M}\,\mat{Z}) = \min_{\v{v}, \v{w}} \Set{\v{v}^\top\v{t} + \v{w}^\top\v{s}}{\mat{R}\Diag(\v{v})\mat{R}^\top + \mat{Q}\Diag(\v{w})\mat{Q}^\top \succeq \mat{M}}
		\end{equation}
	\end{lemma}
	\begin{proof}
		Following the arguments from \cite[Lemma 13.2.2]{Nesterov2000},
		\begin{align*}
			&\max_{\substack{\mat{Z}\succeq 0\\\diag(\mat{R}^\top\mat{Z}\,\mat{R})= \v{t}\\\diag(\mat{Q}^\top\mat{Z}\,\mat{Q}) = \v{s}}} \trace(\mat{M}\,\mat{Z})\\
			&\quad=\max_{\mat{Z}\succeq 0} \;\, \min_{\v{v},\v{w}} \;\, \trace(\mat{M}\,\mat{Z}) + \v{v}^\top(\v{t} - \diag(\mat{R}^\top\mat{Z}\,\mat{R})) + \v{w}^\top(\v{s}-\diag(\mat{Q}^\top\mat{Z}\,\mat{Q}))\\
			&\quad=\max_{\mat{Z}\succeq 0} \;\, \min_{\v{v},\v{w}} \;\, \trace(\mat{M}\,\mat{Z}) - \trace(\mat{R}\Diag(\v{v})\mat{R}^\top\mat{Z}) - \trace(\mat{Q}\Diag(\v{w})\mat{Q}^\top\mat{Z}) + \v{v}^\top\v{t} + \v{w}^\top\v{s}\\
		\end{align*}
		We can now apply Lagrange duality, i.e., \cite[Theorem 1, p. 224]{Luenberger1969} (using $x \equiv (\v{v},\v{w})$ and $f(\v{v},\v{w}) \equiv \v{v}^\top\v{t} + \v{w}^\top\v{s}$, with $G(\v{v},\v{w}) = \mat{M} - \mat{R}\Diag(\v{v})\mat{R}^\top - \mat{Q}\Diag(\v{w})\mat{Q}^\top$ and $\theta \equiv \mat{0}$. The relation $\geq$ in \cite[Theorem 1, p. 224]{Luenberger1969} then corresponds to the semidefinite relation $\succeq$, and $\varphi$ -- which is given in \cite[Eq. (2), p. 223]{Luenberger1969} -- corresponds to the minimum over $\v{v}$ and $\v{w}$). This results in
		\begin{align*}
			&\max_{\substack{\mat{Z}\succeq 0\\\diag(\mat{R}^\top\mat{Z}\,\mat{R})= \v{t}\\\diag(\mat{Q}^\top\mat{Z}\,\mat{Q}) = \v{s}}} \trace(\mat{M}\,\mat{Z})\\
			&\quad = \inf_{\v{v}, \v{w}} \Set{\v{v}^\top\v{t} + \v{w}^\top\v{s}}{\mat{M} - \mat{R}\Diag(\v{v})\mat{R}^\top - \mat{Q}\Diag(\v{w})\mat{Q}^\top \preceq \mat{0}}\\
			&\quad = \inf_{\v{v}, \v{w}} \Set{\v{v}^\top\v{t} + \v{w}^\top\v{s}}{\mat{R}\Diag(\v{v})\mat{R}^\top + \mat{Q}\Diag(\v{w})\mat{Q}^\top \succeq \mat{M}}.
		\end{align*}
		Note that the infimum can be replaced by a minimum, since the constraint set $\Set{(\v{v},\v{w}) \in \reals^l \times \reals^k}{\mat{R}\Diag(\v{v})\mat{R}^\top + \mat{Q}\Diag(\v{w})\mat{Q}^\top \succeq \mat{M}}$ is closed.
	\end{proof}
	We now have all the necessary tools to construct an approximation:
	\begin{theorem}
		\label{thm:main_q<2<p}
		Let $\mat{A} \in \reals^{m\times n}$, $\mat{B}\in \reals^{l\times n}$, and assume $\mat{B}$ is injective. Then, for $1 < q \leq 2 \leq p < \infty$,
		\begin{equation}
			\label{eq:main_thm_q<2<p}
			\norm{\mat{A}}_{p\mapsto q;\mat{B}} \leq \frac{1}{2}\min_{(\v{v},\v{w})\in K} \left(\norm{\v{v}}_{\frac{q}{2-q}}+\norm{\v{w}}_{\frac{p}{p-2}}\right) \leq \gamma_p \gamma_{q^*}\norm{\mat{A}}_{p\mapsto q;\mat{B}},
		\end{equation}
		where
		\begin{equation}
			\label{eq:K_def}
			K := \Set{(\v{v},\v{w}) \in \reals^{m}\times \reals^{l}}{\begin{bmatrix}\Diag(\v{v}) & -\mat{A}\\-\mat{A}^\top & \mat{B}^\top\Diag(\v{w})\mat{B}\end{bmatrix} \succeq 0},
		\end{equation}
		and $\gamma_s = \left(\frac{2^{s/2}}{\sqrt{\pi}}\Gamma\left(\frac{s+1}{2}\right)\right)^{1/s}$ for $s\in[1,\infty)$.
	\end{theorem}
	\begin{proof}	
		Before we prove \eqref{eq:main_thm_q<2<p}, we first show	
		\begin{equation}
			\label{eq:main_thm_q<2<p_original}
			\norm{\mat{A}}_{p\mapsto q;\mat{B}} \leq \max_{\substack{\mat{X} = \mat{U}^\top\mat{V}\\\norm{\mat{V}}_{L_{2,q^*}}\leq 1\\\norm{\mat{B}\,\mat{U}^\top}_{L_{2,p}^\top}\leq 1}} \trace(\mat{A}\,\mat{X}) \leq \gamma_p\gamma_{q^*}\norm{\mat{A}}_{p\mapsto q;\mat{B}},
		\end{equation}
		where the maximum in \eqref{eq:main_thm_q<2<p_original} is taken over all decompositions $\mat{X} = \mat{U}^\top\mat{V}$, with $\mat{U} \in \reals^{k\times n}$ and $\mat{V} \in \reals^{k \times m}$ for some $k\in \naturals$.
		
		\noindent\textbf{Step 1: The first inequality}
		
		We begin with the first inequality of \eqref{eq:main_thm_q<2<p_original}.
		This follows by choosing $k=1$, $\mat{U} = \v{\mu}^\top$, and $\mat{V} = \v{\nu}^\top$ for row vectors $\v{\mu} \in \reals^{n}$ and $\v{\nu} \in \reals^{m}$:
		\begin{align*}
			\max_{\substack{\mat{X} = \mat{U}^\top\mat{V}\\\norm{\mat{V}}_{L_{2,q^*}}\leq 1\\\norm{\mat{B}\,\mat{U}^\top}_{L_{2,p}^\top}\leq 1}} \trace(\mat{A}\,\mat{X})
			\geq \max_{\substack{\norm{\v{\nu}}_{q^*}\leq 1\\\norm{\mat{B}\v{\mu}}_{p}\leq 1}} \v{\nu}^\top\mat{A}\v{\mu}
			= \norm{\mat{A}}_{p\mapsto q;\mat{B}}.
		\end{align*}
		
		\medskip
		
		\noindent\textbf{Step 2: Probabilistic relaxation}
		
		We turn to the second inequality of \eqref{eq:main_thm_q<2<p_original}.	
		Choose an arbitrary $k\in \naturals$, and let $\v{u}_i, \v{v}_j \in \reals^k$ be vectors for $i=1,...,n$ and $j=1,...,m$, and let $\mat{U}$ and $\mat{V}$ be the matrices with columns $\v{u}_i$ and $\v{v}_j$, respectively. For the rest of this step, $\v{u}_i$ and $\v{v}_j$ can be arbitrary, but can not be chosen such that all $\v{u}_i$ are zero or all $\v{v}_j$ are zero (in other words, $\mat{U}$ and $\mat{V}$ may not be the zero matrix).
		By duality,
		\begin{align*}
			\norm{\mat{A}}_{p\mapsto q;\mat{B}} &= \max_{\norm{\mat{B}\v{x}}_p\leq 1}\nnorm{\mat{A}\v{x}}_q
			=\max_{\substack{\norm{\v{y}}_{q^*}\leq 1\\\norm{\mat{B}\v{x}}_p\leq 1}}\v{y}^\top\mat{A}\v{x}
			=\max_{\v{x}\neq \v{0}, \v{y}\neq \v{0}}\frac{\sum_{ij}A_{ji}x_iy_j}{\norm{\mat{B}\v{x}}_p\norm{\v{y}}_{q^*}}.
		\end{align*}
		We use a similar technique to \cite{Guruswami2016} and \cite{Bhattiprolu2018}: Let $\v{\grand}$ denote a standard Gaussian random vector in $\reals^{k}$.
		We replace $(\v{x},\v{y})$ by $(\mat{U}^\top\v{\grand}, \mat{V}^\top\v{\grand})$ according to Lemma \ref{lmm:vijay} to obtain
		\begin{equation}
			\label{eq:numerator_denumerator_q<2<p}
			\norm{\mat{A}}_{p\mapsto q;\mat{B}} \geq \frac{\mean_{\v{\grand}}\left[\sum_{ij}A_{ji}\v{u}_i^\top\v{\grand}\v{\grand}^\top\v{v}_j\right]}{\mean_{\v{\grand}}\left[\norm{\mat{B}\,\mat{U}^\top\v{\grand}}_p \norm{\mat{V}^\top\v{\grand}}_{q^*}\right]}.
		\end{equation}
		We compute the numerator and the denominator of \eqref{eq:numerator_denumerator_q<2<p} separately. The numerator is simple, since $\mean_{\v{\grand}}\left[\v{\grand}\v{\grand}^\top\right] = \mat{I}$ we have
		\begin{equation*}
			\mean_{\v{\grand}}\big[\sum_{ij}A_{ji}\v{u}_i^\top\v{\grand}\v{\grand}^\top\v{v}_j\big] = \sum_{ij}A_{ji}\v{u}_i^\top\v{v}_j = \trace(\mat{A}\,\mat{U}^\top\mat{V}).
		\end{equation*}
		For the denominator, since $q \leq 2 \leq p$ there holds $1/p + 1/q^* \leq 1$, so we may use the Hölder inequality:
		\begin{equation*}
			\mean_{\v{\grand}}\left[\norm{\mat{B}\,\mat{U}^\top\v{\grand}}_p \norm{\mat{V}^\top\v{\grand}}_{q^*}\right] \leq \left(\mean_{\v{\grand}}\left[\norm{\mat{B}\,\mat{U}^\top\v{\grand}}_p^p\right]\right)^{1/p} \left(\mean_{\grand}\left[\norm{\mat{V}^\top\v{\grand}}_{q^*}^{q^*}\right]\right)^{1/q^*}.
		\end{equation*}
		Let $\v{\beta}_i$ denote the rows of $\mat{B}$ for $i=1,...,l$. Using Lemma \ref{lmm:E_ugamma},
		\begin{align*}
			\mean_{\v{\grand}}\left[\norm{\mat{B}\,\mat{U}^\top\v{\grand}}_p^p\right] &= \sum_i\mean_{\v{\grand}}\left[|\v{\beta}_i^\top\mat{U}^\top\v{\grand}|^p\right]
			=\sum_i \gamma_p^p\norm{\mat{U}\v{b}_i}_2^p
			=\gamma_p^p\norm{\mat{B}\,\mat{U}^\top}_{L_{2,p}^\top}^p
		\end{align*} 
		and similarly
		\begin{align*}
			\mean_{\grand}\left[\norm{\mat{V}^\top\v{\grand}}_{q^*}^{q^*}\right] &= \sum_j \gamma_{q^*}^{q^*}\norm{\v{v}_j}_2^{q^*}
			= \gamma_{q^*}^{q^*}\norm{\mat{V}}_{L_{2,q^*}}^{q*}.
		\end{align*}
		Putting everything together, we obtain the second inequality of \eqref{eq:main_thm_q<2<p_original}.
		
		\medskip
		
		\noindent\textbf{Step 3: Positive Semidefinite Reformulation}
		
		As in the previous step, let $\v{\beta}_i$ denote the $i$-th row of $\mat{B}$, and $\v{v}_j$ the $j$-th column of $\mat{V}$. Then
		\begin{align*}
			\max_{\substack{\mat{X} = \mat{U}^\top\mat{V}\\\norm{\mat{V}}_{L_{2,q^*}}\leq 1\\\norm{\mat{B}\,\mat{U}^\top}_{L_{2,p}^\top}\leq 1}} \trace(\mat{A}\,\mat{X})
			&=\max_{\substack{\left(\sum_j (\v{v}_j^\top\v{v}_j)^{q^*/2}\right)^{1/q^*}\leq 1\\\left(\sum_i(\v{\beta}_i^\top\mat{U}^\top\mat{U}\v{\beta}_i)^{p/2}\right)^{1/p}\leq 1}} \frac{1}{2}\trace\left(\begin{bmatrix}\mat{0} & \mat{A}\\\mat{A}^\top & \mat{0}\end{bmatrix}\begin{bmatrix}\mat{V}^\top\\\mat{U}^\top\end{bmatrix}\begin{bmatrix}\mat{V} & \mat{U}\end{bmatrix}\right)\\
			&=\max_{\substack{\nnorm{\v{t}}_{q^*/2} \leq 1\\ \nnorm{\v{s}}_{p/2} \leq 1}} \;\, \max_{\substack{Z\succeq0\\\diag\left(\mat{P}_{\mat{V}}^\top\mat{Z}\,\mat{P}_{\mat{V}}\right) = \v{t}\\\diag\left(\mat{B}\,\mat{P}_{\mat{U}}^\top\mat{Z}\,\mat{P}_{\mat{U}}\mat{B}^\top\right) = \v{s}}}  \;\, \frac{1}{2}\trace\left(\begin{bmatrix}\mat{0} & \mat{A}\\\mat{A}^\top & \mat{0}\end{bmatrix}\mat{Z}\right)
		\end{align*}
		where
		\begin{equation}
			\mat{P}_{\mat{V}} := \begin{bmatrix}\mat{I}_{m}\\\mat{0}_{n\times m}\end{bmatrix},
			\qquad
			\mat{P}_{\mat{U}} := \begin{bmatrix}\mat{0}_{m\times n}\\\mat{I}_n\end{bmatrix}.
		\end{equation}
		Using Lemma \ref{lmm:semi-definite_dual} yields
		\begin{equation}
			\label{eq:almost_final_q<2<p}
			\max_{\substack{\mat{X} = \mat{U}^\top\mat{V}\\\norm{\mat{V}}_{L_{2,q^*}}\leq 1\\\norm{\mat{B}\,\mat{U}^\top}_{L_{2,p}^\top}\leq 1}} \trace(\mat{A}\,\mat{X})
			=\frac{1}{2}\max_{\substack{\nnorm{\v{t}}_{q^*/2} \leq 1\\ \nnorm{\v{s}}_{p/2} \leq 1}} \;\, \min_{(\v{v},\v{w})\in K} \;\, \left( \v{v}^\top\v{t}+\v{w}^\top\v{s}\right),
		\end{equation}
		where $K$ is the set of pairs $(\v{v},\v{w}) \in \reals^{m}\times \reals^{l}$ such that
		\begin{align*}
			&\begin{bmatrix}\mat{I}\\\mat{0}\end{bmatrix}\Diag(\v{v})\begin{bmatrix}\mat{I}&\mat{0}\end{bmatrix}
			+\begin{bmatrix}\mat{0}\\\mat{I}\end{bmatrix}\mat{B}^\top\Diag(\v{w})\mat{B}\begin{bmatrix}\mat{0}&\mat{I}\end{bmatrix}
			\succeq \begin{bmatrix}\mat{0} & \mat{A}\\\mat{A}^\top & \mat{0}\end{bmatrix}\\
			\Leftrightarrow\quad&\begin{bmatrix}\Diag(\v{v}) & \mat{0}\\\mat{0} &\mat{0}\end{bmatrix}
			+\begin{bmatrix}\mat{0} & \mat{0}\\\mat{0} &\mat{B}^\top\Diag(\v{w})\mat{B}\end{bmatrix}
			\succeq \begin{bmatrix}\mat{0} & \mat{A}\\\mat{A}^\top & \mat{0}\end{bmatrix}\\
			\Leftrightarrow\quad&\begin{bmatrix}\Diag(\v{v}) & -\mat{A}\\-\mat{A}^\top & \mat{B}^\top\Diag(\v{w})\mat{B}\end{bmatrix} \succeq 0.
		\end{align*}%
		Since $q \leq 2 \leq p$, the functions $\nnorm{\v{t}}_{q^*/2}$ and $\norm{\v{s}}_{p/2}$ are convex, and thus by a simple application of the minimax theorem (specifically, \cite[Corollary 3.3]{Sion1958}), we can swap the maximum and minimum in \eqref{eq:almost_final_q<2<p} to obtain \eqref{eq:main_thm_q<2<p} with $K$ as defined in \eqref{eq:K_def}.
	\end{proof}
	
	\begin{remark}
		\label{rmk:comparison_vijay}
		For $\mat{B} = \mat{I}$, Theorem \ref{thm:main_q<2<p} yields a $\gamma_p \gamma_{q^*}$-approximation to the $p\mapsto q$-norm for $1 < q \leq 2 \leq p < \infty$. On the other hand, \cite[Theorem 1.4.]{Bhattiprolu2023} showed the $p\mapsto q$-norm is $\NP$-hard to approximate within a factor $1/(\gamma_{p^*}\gamma_q) - \varepsilon$ for any $\varepsilon > 0$, while proposing a $C/(\gamma_{p^*}\gamma_q)$-approximation in \cite{Bhattiprolu2018} for a constant $C \leq 1.00863/\ln(1+\sqrt{2})$ (in fact, our approximation from Theorem \ref{thm:main_q<2<p} coincides with that of \cite{Bhattiprolu2018} for $\mat{B} = \mat{I}$). One can verify $1/\gamma_{s^*} \leq \gamma_s$ for $2 \leq s < \infty$, with equality for $s=2$. Therefore, our bound is consistent with the hardness result from \cite{Bhattiprolu2023}. For most cases, our bound is worse than that of \cite{Bhattiprolu2018}, except for values of $p,q^* \in [2, \infty)$ in a small neighborhood around $(p, q^*) = (2,2)$. 
	\end{remark}
	\subsection{The Case \texorpdfstring{$p\in[2, \infty)$}{p ∈ [2,oo)]}, \texorpdfstring{$q=1$}{q = 1}}
	It was shown in \cite[Theorem 1.1.]{Bhattiprolu2023} that computing $\norm{\mat{A}}_{\infty \mapsto s}$ is $\APX$-hard for $s\in(1,2]$. Since $\norm{\mat{A}}_{\infty \mapsto s} = \norm{\mat{A}^\top}_{s^* \mapsto 1}$ by \cite[Proposition 1.2]{Steinberg2005}, it follows that computing $\norm{\mat{A}}_{p \mapsto 1}$ is $\APX$-hard for $p\in [2,\infty)$. This suggests, again, that the best we can hope for is an approximation. We will find one using the same tools as for the case $1<q\leq 2\leq p<\infty$:
	\begin{theorem}
		\label{thm:main_q=1_2<p}
		Let $\mat{A} \in \reals^{m\times n}$, $\mat{B}\in \reals^{l\times n}$, and assume $\mat{B}$ is injective. Then, for any $p \in [2,\infty)$ (and $q=1$),
		\begin{equation}
			\label{eq:main_thm_q=1_2<p}
			\norm{\mat{A}}_{p\mapsto 1;\mat{B}} \leq \frac{1}{2}\min_{(\v{v},\v{w})\in K} \left(\norm{\v{v}}_{1}+\norm{\v{w}}_{\frac{p}{p-2}}\right) \leq \frac{\gamma_p}{\gamma_1} \norm{\mat{A}}_{p\mapsto q;\mat{B}},
		\end{equation}
		where
		\begin{equation}
			K := \Set{(\v{v},\v{w}) \in \reals^{m}\times \reals^{l}}{\begin{bmatrix}\Diag(\v{v}) & -\mat{A}\\-\mat{A}^\top & \mat{B}^\top\Diag(\v{w})\mat{B}\end{bmatrix} \succeq 0},
		\end{equation}
		and $\gamma_s = \left(\frac{2^{s/2}}{\sqrt{\pi}}\Gamma\left(\frac{s+1}{2}\right)\right)^{1/s}$ for $s\in[1,\infty)$.
	\end{theorem}
	\begin{proof}
		The proof is identical to that of Theorem \ref{thm:main_q<2<p}, once we establish
		\begin{equation}
			\label{eq:main_thm_q=1_2<p_original}
			\max_{\substack{\mat{X} = \mat{U}^\top\mat{V}\\\norm{\mat{V}}_{L_{2,\infty}}\leq 1\\\norm{\mat{B}\,\mat{U}^\top}_{L_{2,p}^\top}\leq 1}} \trace(\mat{A}\,\mat{X}) \leq \frac{\gamma_p}{\gamma_1}\norm{\mat{A}}_{p\mapsto 1;\mat{B}}.
		\end{equation}
		Similarly to the proof of Theorem \ref{thm:main_q<2<p}, we choose an arbitrary $k\in \naturals$ and vectors $\v{u}_i, \v{v}_j \in \reals^k$ that are the columns of matrices $\mat{U}$ and $\mat{V}$, which we assume to be different from the zero matrix. We again use the fact that
		\begin{equation}
			\norm{\mat{A}}_{p\mapsto 1;\mat{B}} = \max_{\v{x}\neq \v{0}, \v{y}\neq \v{0}}\frac{\sum_{ij}A_{ji}x_iy_j}{\norm{\mat{B}\v{x}}_p\norm{\v{y}}_{\infty}},
		\end{equation}
		and using a standard Gaussian random variable $\v{\grand}$ in $\reals^{k}$, we replace $\v{x}$ by $\mat{U}^\top\v{\grand}$ according to Lemma \ref{lmm:vijay}. However we do not yet replace $\v{y}$. This gives us
		\begin{equation}
			\label{eq:numerator_denumerator_q=1_2<p}
			\norm{\mat{A}}_{p\mapsto q;\mat{B}} \geq \frac{\mean_{\v{\grand}}\left[\max_{\v{y}\neq \v{0}}\frac{\sum_{ij}A_{ji}\v{u}_i^\top\v{\grand}\cdot y_j}{\norm{\v{y}}_{\infty}}\right]}{\mean_{\v{\grand}}\left[\norm{\mat{B}\,\mat{U}^\top\v{\grand}}_p\right]}.
		\end{equation}
		The (main) denominator can be bounded as in the proof of Theorem \ref{thm:main_q<2<p}:
		\begin{equation*}
			\mean_{\v{\grand}}\left[\norm{\mat{B}\,\mat{U}^\top\v{\grand}}_p\right] \leq \gamma_p\norm{\mat{B}\,\mat{U}^\top}_{L_{2,p}^\top}.
		\end{equation*}
		As for the numerator, we replace $y_j$ by $\sign(\v{v}_j^\top\v{\grand})$ (note that we do not use Lemma \ref{lmm:vijay}, just the regular definition of the maximum), which yields
		\begin{equation*}
			\mean_{\v{\grand}}\left[\max_{\v{y}\neq \v{0}}\frac{\sum_{ij}A_{ji}\v{u}_i^\top\v{\grand} \cdot y_j}{\norm{\v{y}}_{\infty}}\right] \geq \sum_{ij}A_{ji}\mean_{\v{\grand}}\left[\v{u}_i^\top\v{\grand}\sign(\v{v}_j^\top\v{\grand})\right].
		\end{equation*}
		For $\v{v}_j \neq \v{0}$, since $\sign(\frac{\v{v}_j^\top}{\norm{\v{v}_j}_2}\v{\grand}) = \sign(\v{v}_j^\top\v{\grand})$, according to \cite[Equation (4.3)]{Alon2006},
		\begin{equation}
			\label{eq:alon_naor}
			\mean_{\v{\grand}}\left[\v{u}_i^\top\v{\grand}\sign(\v{v}_j^\top\v{\grand})\right] = \sqrt{2/\pi} \cdot\frac{\v{u}_i^\top\v{v}_j}{\norm{\v{v}_j}_2} \geq \sqrt{2/\pi} \cdot\frac{\v{u}_i^\top\v{v}_j}{\norm{\mat{V}}_{L_{2,\infty}}}.
		\end{equation}
		Note that
		\begin{equation}
			\label{eq:alon_naor_generalized}
			\mean_{\v{\grand}}\left[\v{u}_i^\top\v{\grand}\sign(\v{v}_j^\top\v{\grand})\right] \geq \sqrt{2/\pi} \cdot\frac{\v{u}_i^\top\v{v}_j}{\norm{\mat{V}}_{L_{2,\infty}}}.
		\end{equation}
		still holds for $\v{v}_j = \v{0}$.	Since $\gamma_1 = \sqrt{2/\pi}$, this results in the bound from \eqref{eq:main_thm_q=1_2<p_original}.
	\end{proof}
	
	\begin{remark}
		Using the same arguments as in Remark \ref{rmk:comparison_vijay}, Theorem \ref{thm:main_q=1_2<p} yields a $\gamma_p/\gamma_1$-approximation for the $p\mapsto 1$-norm of a matrix (and the algorithm coincides with that of \cite{Bhattiprolu2018}). This bound is better than the one in \cite{Bhattiprolu2018} for $p\lessapprox 3.98969$. In particular, for $p=2$ our bound is optimal according to \cite[Theorem 1.4.]{Bhattiprolu2023}, unless $\Poly = \NP$.
	\end{remark}
	
	\subsection{The Case \texorpdfstring{$p\in [1,2)$}{p ∈ [1,2)}, \texorpdfstring{$q = 1$}{q = 1}}
	It was proven in \cite[Theorem 6.4.]{Bhaskara2011} that $\norm{\mat{A}}_{\infty \mapsto q}$ for $q\in (2,\infty)$ is not approximable in polynomial time within any constant factor, unless $\Poly = \NP$. By duality, this means $\norm{\mat{A}}_{p \mapsto 1}$ is not approximable for $p\in (1, 2)$. Moreover, we prove in \cite{Kulmburg2024} that $\norm{\mat{A}}_{1 \mapsto 1; \mat{B}}$ is not approximable, unless $\NP = \problem{RP}$. Therefore, we can only hope for algorithms that have a non-constant approximation ratio. To deduce one such approximation, we will use the Kahane contraction principle:
	\begin{theorem}[Kahane contraction principle (see {\cite[p. 51, Theorem 12.1]{Schwartz1981}})]
		\label{thm:kahane}
		For $k\in\naturals$ let $\mathsf{v}_1,...,\mathsf{v}_k$ be symmetric, independent random variables in a Banach space $\Vset$ with norm $\norm$, and let $c_1,\cdots, c_k \in \reals$. Then, for any $p\geq 1$,
		\begin{equation}
			\mean_{\mathsf{v}_1,...,\mathsf{v}_k}\left[\norm{c_1\mathsf{v}_1+\cdots+c_k\mathsf{v}_k}^p\right] \leq \max_i|c_i|^p\;\mean_{\mathsf{v}_1,...,\mathsf{v}_k}\left[\norm{\mathsf{v}_1+\cdots+\mathsf{v}_k}^p\right].
		\end{equation}
	\end{theorem}
	Using the Kahane contraction principle, we can construct an efficient approximation scheme, even though it is not a constant-ratio approximation:
	\begin{theorem}
		\label{thm:main}
		Let $\mat{A} \in \reals^{m\times n}$, $\mat{B}\in \reals^{l\times n}$, and assume $\mat{B}$ is injective. Then, for any $p\in[1,\infty)$,
		\begin{equation}
			\label{eq:main_thm_q=1}
			\norm{\mat{A}}_{p\mapsto1;\mat{B}} \leq \min_{\mat{B}^\top\mat{Y} = \mat{A}^\top} \norm{\mat{Y}}_{L_{1,p^*}^\top} \leq \frac{\gamma_p\sqrt{m}}{\gamma_1}\norm{\mat{A}}_{p\mapsto1;\mat{B}},
		\end{equation}
		where $\gamma_s = \left(\frac{2^{s/2}}{\sqrt{\pi}}\Gamma\left(\frac{s+1}{2}\right)\right)^{1/s}$ for $s\in[1,\infty)$.
	\end{theorem}
	\begin{proof}
		Similarly to the other proofs of this section, we consider the first and second inequality of \eqref{eq:main_thm_q=1} separately.
		
		\medspace
		
		\noindent\textbf{Step 1: The first inequality}
		
		For any matrix $\mat{Y}$ with rows $\v{\upsilon}_i$, by duality
		\begin{align*}
			\norm{\mat{Y}}_{L_{1,p^*}^\top} 
			= \Big(\sum_i\max_{\norm{\v{\sigma}}_{\infty}\leq 1}|\v{\upsilon}_i^\top\v{\sigma}|^{p^*}\Big)^{1/p^*}
			\geq\max_{\norm{\v{\sigma}}_{\infty}\leq 1}\norm{\mat{Y}\v{\sigma}}_{p^*},
		\end{align*}
		where for the inequality we used the fact that, for any family of functions $f_i(x)$, there holds $\max_x \sum_i f_i(x) \leq \sum_i \max_x f_i(x)$. Additionally, using the property that for any function $f(x,y)$ there holds $\max_x\min_y f(x,y) \leq \min_y \max_x f(x,y)$, we conclude
		\begin{equation*}
			\min_{\mat{B}^\top\mat{Y} = \mat{A}^\top} \;\, \max_{\norm{\v{\sigma}}_{\infty}\leq 1} \;\, \norm{\mat{Y}\v{\sigma}}_{p^*} \geq \max_{\norm{\v{\sigma}}_{\infty}\leq 1} \;\, \min_{\mat{B}^\top\mat{Y} = \mat{A}^\top} \;\, \norm{\mat{Y}\v{\sigma}}_{p^*}.
		\end{equation*}
		It then suffices to apply Proposition \ref{prop:constrained_norm_optimization} to get
		\begin{equation*}
			\min_{\mat{B}^\top\mat{Y} = \mat{A}^\top} \norm{\mat{Y}}_{L_{1,p^*}^\top} \geq \max_{\norm{\v{\sigma}}_{\infty}\leq1} \;\, \max_{\norm{\mat{B}\v{x}}_p\leq 1} \;\, \v{\sigma}^\top\mat{A}\v{x} = \max_{\norm{\mat{B}\v{x}}_p\leq 1} \norm{\mat{A}\v{x}}_1 = \norm{\mat{A}}_{p\mapsto1;\mat{B}}.
		\end{equation*}
		
		\medskip
		
		\noindent\textbf{Step 2: Probabilistic relaxation}
		
		We turn to the second inequality of \eqref{eq:main_thm_q=1}.	
		For $i=1,\hdots,n$, let $\v{u}_i \in \reals^{m}$ be arbitrary but fixed vectors that are not all zero, and let $\mat{U}$ be the matrix with columns $\v{u}_i$.	
		As in the proof of Theorem \ref{thm:main_q<2<p}, we transform $\norm{\mat{A}}_{p\mapsto 1;\mat{B}}$ using duality:
		\begin{align*}
			\norm{\mat{A}}_{p\mapsto 1;\mat{B}} =\max_{\v{x}\neq \v{0}, \v{y} \neq \v{0}}\frac{\sum_{ij}A_{ji}x_iy_j}{\norm{\mat{B}\v{x}}_p\norm{\v{y}}_{\infty}}.
		\end{align*}
		Using the same arguments as in the proof of Theorem \ref{thm:main_q=1_2<p} but with $\mat{V} = \mat{I}_m$, for a standard Gaussian random variable $\v{\grand}$ in $\reals^{m}$ we get
		\begin{equation}
			\label{eq:numerator_denumerator}
			\norm{\mat{A}}_{p\mapsto 1;\mat{B}} \geq \frac{\mean_{\v{\grand}}\left[\max_{\v{y}\neq \v{0}}\frac{\sum_{ij}A_{ji}\v{u}_i^\top\v{\grand}\cdot y_j}{\norm{\v{y}}_{\infty}}\right]}{\mean_{\v{\grand}}\big[\norm{\mat{B}\,\mat{U}^\top\v{\grand}}_p\big]} \geq \frac{\gamma_1\cdot\trace(\mat{A}\,\mat{U}^\top)}{\mean_{\v{\grand}}\big[\norm{\mat{B}\,\mat{U}^\top\v{\grand}}_p\big]}.
		\end{equation}
		
		\noindent\textbf{Step 3: Bounding the denominator}
		
		We start with a standard application of Jensen's inequality since the function $x \rightarrow x^{1/p}$ is concave for $p\geq 1$:
		\begin{align*}
			\mean_{\v{\grand}}\big[\nnorm{\mat{B}\,\mat{U}^\top\v{\grand}}_p\big]  &= \mean_{\v{\grand}} \big[   \sum_i\big|   \sum_j B_{ij}\iprod{\v{u}_j}{\v{\grand}}_{\reals^m}   \big|^p   \big]^{1/p}\\
			&\leq \big[\sum_i\mean_{\v{\grand}}\big[\big|\langle\sum_j B_{ij}\v{u}_j,\v{\grand}\rangle_{\reals^m}\big|^p\big]\big]^{1/p}.
		\end{align*}
		Using the Kahane contraction principle on $\mean_{\v{\grand}}\big[\big|\langle\sum_j B_{ij}\v{u}_j,\v{\grand}\rangle_{\reals^m}\big|^p\big]$ gives us
		\begin{equation*}
			\mean_{\v{\grand}}\big[\big|\langle\sum_j B_{ij}\v{u}_j,\v{\grand}\rangle_{\reals^m}\big|^p \big]\leq \max_k\big|\sum_jB_{ij}u_{j,k}\big|^p\mean_{\v{\grand}}\left[|\grand_1+\cdots+\grand_{m}|^p\right].
		\end{equation*}
		Since $\grand_i \sim \mathcal{N}(0,1)$, we have $\grand_1+\cdots\grand_{m} \sim \mathcal{N}(0,\sqrt{m})$, and thus $\mean_{\v{\grand}}\left[|\grand_1+\cdots+\grand_{m}|^p\right] = \sqrt{m}^p\gamma_p^p$.
		Overall, we obtain
		\begin{equation*}
			\mean_{\v{\grand}}\left[\nnorm{\mat{B}\,\mat{U}^\top\v{\grand}}_p\right] \leq \gamma_p\sqrt{m}\cdot\big(\sum_i \max_k \big|\sum_j B_{ij}u_{j,k}\big| \big)^{1/p} = \gamma_p\sqrt{m}\cdot\norm{\mat{B}\mat{U}^\top}_{L_{\infty,p}^\top}.
		\end{equation*}
		
		\noindent\textbf{Step 4: Final transformation}
		
		Combining the results from Steps 2 and 3,
		\begin{equation*}
			\max_{\nnorm{\mat{B}\,\mat{X}}_{L_{\infty,p}^\top}\leq1} \trace(\mat{A}\mat{X}) \leq \frac{\gamma_p\sqrt{m}}{\gamma_1}\cdot \norm{\mat{A}}_{p\mapsto1;\mat{B}}.
		\end{equation*}
		It now suffices to use Proposition \ref{prop:constrained_norm_optimization} to transform the left-hand side:
		\begin{equation}
			\max_{\norm{\mat{B}\,\mat{X}}_{L_{\infty,p}^\top}\leq 1} \trace(\mat{A}\,\mat{X}) = \min_{\mat{B}^\top\mat{Z} = \mat{A}^\top} \norm{\mat{Z}}_{L_{1,p^*}^\top},
		\end{equation}
		which completes the proof.
	\end{proof}
	\begin{remark}
		Note that we have stated Theorem \ref{thm:main} for $p\in[1,\infty)$, even though for $p\in [2, \infty)$ a better approximation is available through Theorem \ref{thm:main_q=1_2<p}. This is because the approximation of Theorem \ref{thm:main} is easier to evaluate, which has applications as discussed in \cite{Kulmburg2024}.
	\end{remark}
	Besides the result of Theorem \ref{thm:main}, a different approximation is available for $p\in[1,2)$ and $q=1$, based on the approximation of Theorem \ref{thm:main_q=1_2<p}:
	\begin{corollary}
		\label{cor:main_nesterov}
		Let $\mat{A} \in \reals^{m\times n}$, $\mat{B}\in \reals^{l\times n}$, and assume $\mat{B}$ is injective. Then, for any $p\in[1,2)$,
		\begin{equation}
			\label{eq:main_cor_q=1}
			\norm{\mat{A}}_{p\mapsto1;\mat{B}} \leq \frac{1}{2}\min_{(\v{v},\v{w})\in K} \left(\norm{\v{v}}_{1}+\norm{\v{w}}_{\infty}\right) \leq \sqrt{\pi/2}\cdot l^{1/p-1/2}\norm{\mat{A}}_{p\mapsto1;\mat{B}},
		\end{equation}
		where
		\begin{equation}
			\label{eq:redef_K}
			K := \Set{(\v{v},\v{w}) \in \reals^{m}\times \reals^{l}}{\begin{bmatrix}\Diag(\v{v}) & -\mat{A}\\-\mat{A}^\top & \mat{B}^\top\Diag(\v{w})\mat{B}\end{bmatrix} \succeq 0}.
		\end{equation}
	\end{corollary}
	\begin{proof}
		If $p\in[1,2)$, \eqref{eq:main_cor_q=1} follows directly from Theorem \ref{thm:main_q=1_2<p} (for the case $p=2$) and the fact that for any vector $\v{v}\in\reals^{l}$, $\norm{\v{v}}_2 \leq \norm{\v{v}}_p \leq l^{1/p-1/2} \norm{\v{v}}_2$:
		\begin{align*}
			\norm{\mat{A}}_{p\mapsto1;\mat{B}}
			&\leq \norm{\mat{A}}_{2\mapsto1;\mat{B}} && \text{(since $\norm{\v{v}}_2 \leq \norm{\v{v}}_p$)}\\
			&\leq \frac{1}{2}\min_{(\v{v},\v{w})\in K} \left(\norm{\v{v}}_{1}+\norm{\v{w}}_{\infty}\right) && \text{(using Theorem \ref{thm:main_q=1_2<p})}\\
			&\leq \sqrt{\pi/2}\cdot\norm{\mat{A}}_{2\mapsto1;\mat{B}}&& \text{(using Theorem \ref{thm:main_q=1_2<p})}\\
			&\leq \sqrt{\pi/2}\cdot l^{1/p-1/2}\norm{\mat{A}}_{p\mapsto1;\mat{B}} && \text{(since $\norm{\v{v}}_p \leq l^{1/p-1/2} \norm{\v{v}}_2$)},
		\end{align*}
		with $K$ as defined in \eqref{eq:redef_K}.
	\end{proof}
	\begin{remark}
		The approximations of Corollary \ref{cor:main_nesterov} should be handled with care; while they may scale better than Theorem \ref{thm:main} for $p\in(1,2)$, if $l=m$ and $p=1$, the approximation ratio from Theorem \ref{thm:main} is always better than that of Corollary \ref{cor:main_nesterov}. For $p\in(1,2)$, this depends on the values of $l$ and $p$. Additionally, the approximations of Theorem \ref{thm:main} can typically be computed much faster than that of Corollary \ref{cor:main_nesterov}, since they are convex optimization problems with linear equality constraints, whereas the approximations from Corollary \ref{cor:main_nesterov} require solving a semidefinite program, which is more time consuming in practice.
	\end{remark}
	
	\section{Experimental Results}
	\label{sec:numerical_evaluations}
	
	To evaluate the accuracy of the approximations we have derived in Section \ref{sec:approximable}, we examine their performance numerically. We do so only for $q=1$, as in this case Corollary \ref{cor:matrix_norm_and_optimization} implies
	\begin{equation}
		\norm{\mat{A}}_{p\mapsto 1;\mat{B}} = \max_{\norm{\v{\alpha}}_{\infty}\leq 1} \;\, \min_{\mat{B}^\top\v{\beta} = \mat{A}^\top\v{\alpha}} \;\, \nnorm{\v{\beta}}_{p^*}.
	\end{equation}
	The term
	\begin{equation}
		\label{eq:loupe_term}
		\min_{\mat{B}^\top\v{\beta} = \mat{A}^\top\v{\alpha}}\nnorm{\v{\beta}}_{p^*}
	\end{equation}
	can be evaluated in polynomial time using interior point methods (see \cite[Chapter 11]{Boyd2004}) since it is an equality-constrained convex minimization problem. Furthermore, \eqref{eq:loupe_term} is convex in $\v{\alpha}$, so by the Bauer maximum principle the exact value of $\norm{\mat{A}}_{p\mapsto 1;\mat{B}}$ can be computed by taking the maximum of \eqref{eq:loupe_term} over all $\v{\alpha}\in \{-1,+1\}^{m}$. A similar approach would not work for $q > 1$, since the set $\{\v{\alpha}\in\reals^{m}\,|\,\norm{\v{\alpha}}_{q^*}\leq 1\}$ has infinitely many extreme points.
	
	Let $\mathtt{approx}_p(\mat{A}, \mat{B})$ denote the approximation of Theorem \ref{thm:main}, Theorem \ref{thm:main_q=1_2<p}, or Corollary \ref{cor:main_nesterov} for matrices $\mat{A}$ and $\mat{B}$. Then
	\begin{equation*}
		\rho(\mat{A},\mat{B}) := \frac{\mathtt{approx}_p(\mat{A},\mat{B})}{\norm{\mat{A}}_{p\mapsto 1;\mat{B}}}
	\end{equation*}
	is a lower bound of the actual approximation ratio. Using a global optimization solver (in our case, we used the surrogate optimization solver from MATLAB, see \cite{Gutmann2001}), one can search for the maximum of $\rho(\mat{A}, \mat{B})$ as a function of $\mat{A}$ and $\mat{B}$.
	Concretely, we searched for the maximal value $\rho_{\max}$ of $\rho(\mat{A}, \mat{B})$ over the set
	\begin{equation*}
		\Set{(\mat{A},\mat{B})\in\reals^{m\times n}\times \reals^{l\times n}}{\norm{\mat{A}}_{L_{\infty, \infty}}\leq 10, \norm{\mat{B}}_{L_{\infty, \infty}}\leq 10}.
	\end{equation*}
	For each iteration, we verified that $\mat{B}$ was injective; If not, $\rho(\mat{A}, \mat{B})$ was manually set to 1 to not impact the final result. For the cases where $p\neq 1$, we used 1000 iterations for each choice of $l,m,$ and $n$, while for $p=1$, we used 10000 iterations. This is because for $p=1$, computing the exact value of $\norm{\mat{A}}_{1\mapsto 1;\mat{B}}$ can be reformulated as a zonotope containment problem (see \cite{Kulmburg2021}), for which there are faster algorithms, such as the algorithm $\verb|polymax|$ from \cite{Kulmburg2021}, computed using the CORA toolbox \cite{Althoff2015}. 
	
	\begin{figure}[t!]
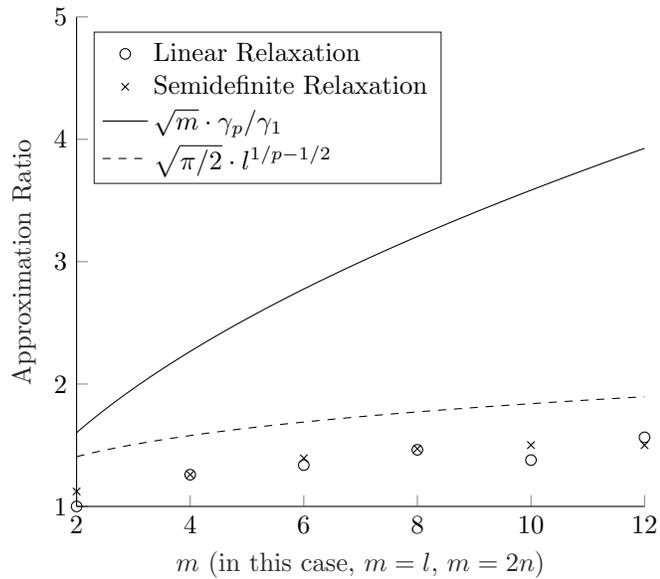

		\centering
%
		\caption{Numerical results for $q=1$ and $p=1.5$. The solid line corresponds to the theoretical worst case approximation ratio from Theorem \ref{thm:main}, while the dashed line corresponds to that of Corollary \ref{cor:main_nesterov}. Both theoretical bounds over-approximate the worst-case approximation ratio we measured experimentally.}
		\label{fig:p_equals_1.5}
	\end{figure}
	\begin{figure}[t!]
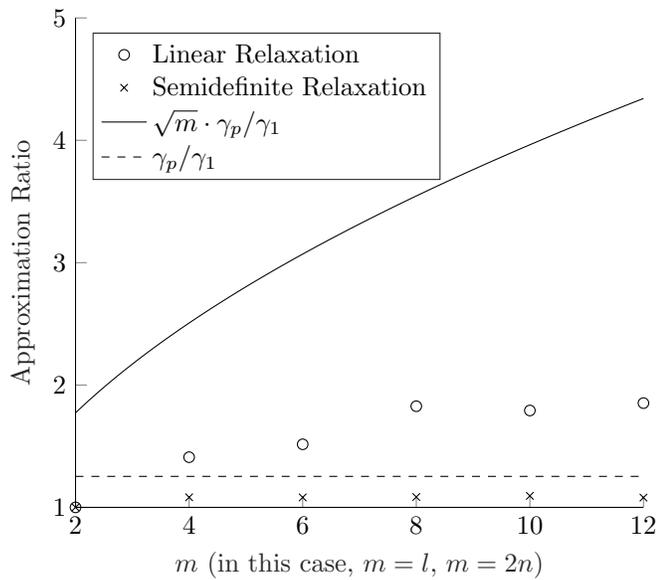

		\centering
		%
%
		\caption{Numerical results for $q=1$ and $p=2$. The semidefinite relaxation yields, as expected, a worst-case approximation ratio better than $\gamma_2/\gamma_1 = \sqrt{\pi/2}$. It also performs significantly better than the linear relaxation.}
		\label{fig:p_equals_2}
	\end{figure}
	
	For $p=1.5$ and $p=2$, we compared the results of the approximation from Theorem \ref{thm:main}, which we call the Linear Relaxation approach, to that of Corollary \ref{cor:main_nesterov} or Theorem \ref{thm:main_q=1_2<p} (depending on $p$), which we call the Semidefinite Relaxation approaches. We chose $l, m,$ and $n$ so that $l = m = 2n$ and carried out the experiments for different values of $n=1,\hdots, 6$. The results can be seen in Fig. \ref{fig:p_equals_1.5} and Fig. \ref{fig:p_equals_2}.
	
	For $p=1$, we instead evaluated only the linear relaxation, for each $m = n,\hdots,20$, with $n=1,\hdots,9$ and $l = m$. In Figure \ref{fig:p_equals_1}, we show the maximum value of $\rho_{\max}$ for fixed values of $m$, over multiple values of $n$.
	
	We used the YALMIP toolbox (see \cite{Lofberg2004}) combined with the MOSEK solver \cite{Mosek2022} to solve the linear and semidefinite optimization problems. The code used to generate Fig. \ref{fig:p_equals_1.5}-\ref{fig:p_equals_1} can be found in the Code Ocean capsule available at the URL
	
	\medskip
	
	\begin{center}
		\url{https://codeocean.com/capsule/8982170/tree}
	\end{center}
	
	\medskip
	
	\subsection{Discussion}
	
	As can be seen in Figures \ref{fig:p_equals_1.5}-\ref{fig:p_equals_1}, the bounds deduced in Theorem \ref{thm:main_q=1_2<p}, Theorem \ref{thm:main}, and Corollary \ref{cor:main_nesterov} hold for all tested samples (i.e., the theoretical worst-case approximation ratio always over-approximates $\rho_{\max}$). However, in some cases our bounds seem to be too conservative. This is particularly apparent for $p=1$. In fact, for higher values of $m$, the value of $\rho_{\max}$ even seems to \emph{drop}, contrary to our expectations. This can be partly explained by the fact that, for higher $l$, $m$, and $n$, the size of the space of possible matrices $\mat{A}$ and $\mat{B}$ considerably increases, so the solver computing $\rho_{\max}$ does not have enough iterations to converge to the actual maximum for $\rho_{\max}$. Due to the high computational cost of evaluating the exact value of $\norm{\mat{A}}_{p\mapsto 1;\mat{B}}$, we could only perform our experiments with relatively low values of $l$, $m$, and $n$. It is possible that, for larger $l$, $m$, and $n$, higher values could be found for $\rho_{\max}$.
	
	\begin{figure}[t!]
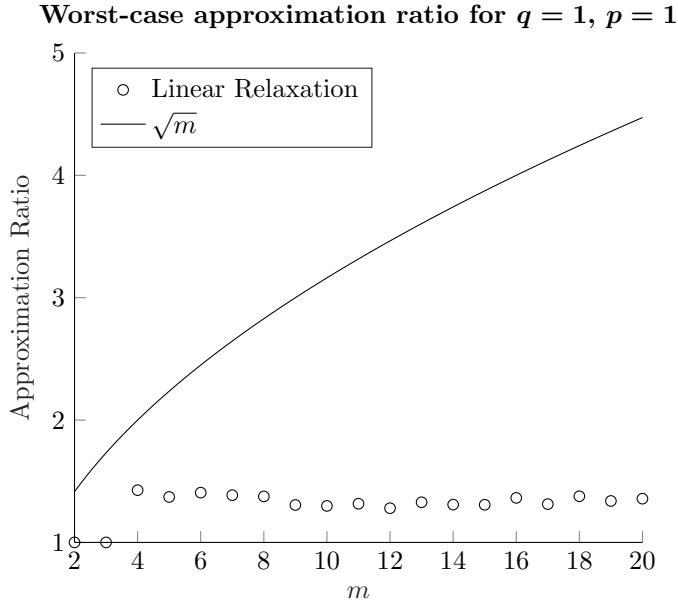

		\centering
%
		\caption{Numerical results for $q=1$ and $p=1$. Since, in this case, the exact result can be formulated as a zonotope containment problem, a larger number of cases could be evaluated.}
		\label{fig:p_equals_1}
	\end{figure}
	
	\section{Conclusions}
	The concepts of push-forward and pull-back of norms, as well as the fact that these are dual to each other, allowed us to reformulate the generalized matrix norm problem. In particular, we discovered that max-min problems involving vector norms could sometimes be reformulated, via duality, as generalized matrix norm problems. This kind of max-min problems occurs naturally in the context of containment problems \cite{Kulmburg2021,Kulmburg2024}, but it would be interesting to investigate whether such problems can also be encountered in other contexts, for example as approximations to more general max-min problems.
	
	We also analyzed the computability of the matrix norm $\norm{\mat{A}}_{p\mapsto q;\mat{B}}$
	for different values of $p$ and $q$. In some cases, we found exact algorithms that run in polynomial time with respect to the representation size of $\mat{A}$ and $\mat{B}$. For cases where such exact algorithms were not available, we have introduced approximations that run in polynomial time. However, in some instances, these approximations seem to perform significantly better than the error bounds we deduced, which leaves open the possibility that more accurate error bounds could be found in the future, especially for the case $q=1$, $p\in[1,2)$.
	
	\appendix
	\section{The Convex Conjugate of a Norm}
	\label{sec:convex_conjugate}
	Let $\Vset$ be a topological vector space, with dual space $\Vset^*$ as discussed in Section \ref{sec:duality}. For a function $\tau : \Vset \rightarrow \reals$, the \emph{convex conjugate} $\tau^\bigstar : \Vset^* \rightarrow \reals$ is
	\begin{equation}
		\tau^\bigstar(f) := \sup_{v \in \Vset} f(v) - \tau(v).
	\end{equation}
	For more information on the convex conjugate and its properties, we refer to \cite[Section 3]{Broendsted1964}.
	\begin{lemma}
		\label{lmm:convex_conjugate_norm}
		Let $(\Vset, \norm_{\Vset})$ be a normed space. Then for $f \in \Vset^*$, the convex conjugate $\norm{f}_{\Vset}^{\bigstar}$ (not to be confused with the dual norm $\norm{f}_{\Vset}^*$) is given by
		\begin{equation}
			\norm{f}_{\Vset}^{\bigstar} = 
			\begin{cases}
				&0, \text{ if } \norm{f}_{\Vset}^* \leq 1,\\
				&\infty, \text{ otherwise}.
			\end{cases}
		\end{equation}
	\end{lemma}
	\begin{proof}
		The proof is nearly identical to \cite[Example 3.26]{Boyd2004}: If $f \in \Vset^*$ with $\norm{f}_{\Vset}^* > 1$, by definition of the dual norm there exists a $v\in\Vset$ such that $\norm{v}_{\Vset} \leq 1$ but $f(v) > 1$. Taking $v' := tv$ for $t \in [0,\infty)$ and letting $t\rightarrow \infty$ we get
		\begin{equation*}
			f(v') - \norm{v'}_{\Vset} = t(f(v) - \norm{v}_{\Vset}) \rightarrow \infty,
		\end{equation*}
		which proves $\norm{f}_{\Vset}^{\bigstar} = \infty$ in this case. If $\norm{f}_{\Vset}^* \leq 1$, the definition of the dual norm implies $f(v) \leq \norm{f}_{\Vset}^* \cdot \norm{v}_{\Vset} \leq \norm{v}_{\Vset}$ for every $v \in \Vset$, so $f(v) - \norm{v}_{\Vset} \leq 0$. Consequently, $v = 0$ maximizes $f(v) - \norm{v}_{\Vset}$, thus $\norm{f}_{\Vset}^{\bigstar} = 0$ in this case.
	\end{proof}

	\section*{Acknowledgments}
	The author would like to thank Mark Wetzlinger and Florian Finkeldei for proofreading this manuscript, Matthias Althoff for his invaluable support throughout the composition of this article, as well as the anonymous reviewers for their meticulous corrections.

\end{document}